\documentclass[10pt]{article}

\usepackage[margin=1.53in]{geometry}

\usepackage{amsthm,amsmath}
\usepackage{amssymb,amscd}
\usepackage{mathrsfs,eucal}
\usepackage{enumerate}
\usepackage{color}
\usepackage[all,cmtip]{xy}
\usepackage{hyperref}
\usepackage{placeins}

\makeatletter
\renewcommand{\section}{\@startsection%
  {section}
  {1}
  {0em}
  {\baselineskip}
  {0.5\baselineskip}
  {\normalfont\large\bfseries}}
\renewcommand{\subsection}{\@startsection%
  {subsection}
  {2}
  {0em}
  {\baselineskip}
  {0.25\baselineskip}
  {\normalfont\bfseries}
}
\makeatother

\usepackage{amsmath,amssymb}
\usepackage{tikz,ifthen} 

\usepackage{bbold}
\usepackage{enumerate}

\usepackage{amsthm}
\newtheorem{theorem}{Theorem}
\newtheorem*{notation*}{Notation}
\newtheorem{lemma}{Lemma}
\newtheorem{corollary}{Corollary}
\newtheorem{definition}{Definition}
\newtheorem{remark}{Remark}
\newtheorem{proposition}{Proposition}
\usepackage{amsmath,amssymb}
\usepackage[ruled,vlined]{algorithm2e}

\begin{document}

\title{\bf Locality in Network Optimization\thanks{Supported in 
part by the NSF Grant ECCS-1609484.}}

\author{Patrick Rebeschini\thanks{P.~Rebeschini was with the Department of Electrical Engineering, Yale University, New Haven, CT 06511, USA. He is now with the Department of Statistics, University of Oxford, Oxford, OX1 3LB, UK (e-mail: patrick.rebeschini@stats.ox.ac.uk).} \and Sekhar Tatikonda\thanks{S.~Tatikonda is with the Department of Statistics and Data Science, Yale University, New Haven, CT 06511, USA (e-mail: sekhar.tatikonda@yale.edu).}}

\date{}

\maketitle

\begin{abstract} 
In probability theory and statistics notions of correlation among random variables, decay of correlation, and bias-variance trade-off are fundamental. In this work we introduce analogous notions in optimization, and we show their usefulness in a concrete setting. We propose a general notion of correlation among variables in optimization procedures that is based on the sensitivity of optimal points upon (possibly finite) perturbations. We present a canonical instance in network optimization (the min-cost network flow problem) that exhibits locality, i.e., a setting where the correlation decays as a function of the graph-theoretical distance in the network. In the case of warm-start reoptimization, we develop a general approach to localize a given optimization routine in order to exploit locality. We show that the localization mechanism is responsible for introducing a bias in the original algorithm, and that the bias-variance trade-off that emerges can be exploited to minimize the computational complexity required to reach a prescribed level of error accuracy. We provide numerical evidence to support our claims.

\end{abstract}

\smallskip


\section{Introduction}
\label{sec:introduction}

Many problems in machine learning, networking, control, and statistics can be naturally posed in the framework of network-structured convex optimization. Given the huge problem size involved in modern applications, a lot of efforts have been devoted to the development and analysis of algorithms where the computation and communication are distributed over the network. 
A crucial challenge remains that of identifying the structural amount of information that each computation node needs to receive from the network to yield an approximate solution with a certain accuracy.

The literature on distributed algorithms in network optimization is gigantic, with much of the earlier seminal work contained in the textbook \cite{BT97}. More recent work that explicitly relate the convergence behavior of the algorithms being considered to the network topology and size include --- but it is certainly not limited to --- \cite{Johansson:2009,MRS10,SAW12} for first-order methods, and \cite{WOJ13a,WOJ13b} for second-order methods. Distributed algorithms are iterative in nature. In their synchronous implementations, at every iteration of the algorithm each computation node processes local information that come from its neighbors. Convergence analysis yields the number of iterations that guarantee these algorithms to meet a prescribed level of error accuracy. In turns, these results translate into bounds on the total amount of information processed by each node, as $k$ iterations of the algorithm means that each node receives information coming from nodes that are at most $k$-hops away in the network. As different algorithms yield different rates of convergence, the bounds on the propagation of information that are so-derived are algorithm-dependent (and also depend on the error analysis being considered). As such, they do not capture structural properties of the optimization problem.

The main aim of this paper is to propose a general notion of ``correlation" among variables in network optimization procedures that fundamentally characterizes the propagation of information across the network. This notion is based on the sensitivity of the optimal solution of the optimization problem upon perturbations of parameters locally supported on the network. This notion can be used to investigate ``locality,'' by which we mean problem instances where the correlation decays as a function of the natural distance in the network, so that, effectively, information only propagates across local portions of it. The phenomenon of locality characterizes situations where local perturbations are confined inside small regions of the network. How small the affected regions are, it depends on the particular problem instance at hand, and on the underlying graph topology. If locality is present, one could hope to exploit it algorithmically. In the case of localized perturbations, for instance, one could hope to localize known optimization routines so that only the affected regions are updated, hence yielding computational savings. While such a localization mechanism would introduce a structural ``bias'' with respect to the original algorithm --- which updates every node in the network and not only the ones that are mostly affected by the perturbation --- one could hope to exploit the ``reduction in the variance'' to achieve a prescribed level of error accuracy with a lower computational complexity. In the following we make all of this precise.

In this paper we define general notions of correlation, decay of correlation (locality), and bias-variance decomposition and trade-off in optimization, and we consider a concrete setting to illustrate their usefulness.
The results that we present should be seen as an implementation of a general agenda that can be followed to establish and exploit locality in more general instances of network optimization.

This paper presents four main contributions discussed in separate sections, which we now summarize. Proofs are in the Appendices.

\textbf{1) Sensitivity of optimal points.} In Section \ref{sec:Local sensitivity for equality-constrained convex minimization} we present a theory on the sensitivity of optimal points for smooth convex problems with linear equality constraints upon (possibly finite) perturbations. We consider the problem of minimizing a smooth convex function $x\rightarrow f(x)$ subject to $Ax = b$, for a certain matrix $A$ and vector $b\in\operatorname{Im}(A)$, where $\operatorname{Im}(A)$ denotes the image of $A$. We consider this problem as a function of the constraint vector $b$. We prove that if $f$ is strongly convex, then the optimal point $b \rightarrow x^\star(b)$ is continuously differentiable along $\operatorname{Im}(A)$, and we explicitly characterize the effect that perturbations have on the optimal solution as a function of the objective function $f$, the constraint matrix $A$ and vector $b$. Given a differentiable function $\varepsilon\in\mathbb{R} \rightarrow b(\varepsilon)\in \operatorname{Im}(A)$, we show that the quantity $\frac{d x^\star(b(\varepsilon))}{d\varepsilon}$ is a function of the Moore-Penrose pseudoinverse of the matrix $A\Sigma(b(\varepsilon))A^T$, where $A^T$ is the transpose of $A$, and where $\Sigma(b)$ denotes the inverse of the Hessian of $f$ evaluated at $x^\star(b)$, namely, $\Sigma(b):=\nabla^2 f(x^\star (b))^{-1}$.
The literature on the sensitivity of optimal points (see Section \ref{sec:Local sensitivity for equality-constrained convex minimization} for a list of references) is typically only concerned with establishing infinitesimal perturbations locally, i.e., on a neighborhood of a certain $b\in\operatorname{Im}(A)$. On the other hand, the results that we present extend to \emph{finite} perturbations as well, as we prove that the derivatives of the optimal point are continuous along $\operatorname{Im}(A)$, and hence can be integrated to deal with finite perturbations.
The workhorse behind our results is Hadamard's global inverse function theorem \cite{krantz2002implicit}, which in our setup yields necessary and sufficient conditions for the inverse of the KKT map to be continuously differentiable.
Proofs are in Appendix \ref{sec:Hadamard global inverse theorem}.\\
\indent\textbf{2) Notions of correlation in optimization.} 
In Section \ref{sec:Notions of correlation in optimization} we provide an interpretation of the sensitivity theory previously developed in terms of notions of correlation among variables in optimization procedures, resembling analogous notions in probability theory. If the matrix $A$ is full row rank, for instance, then the quantity $\frac{\partial x^\star(b)_i}{\partial b_a}$ is well-defined and describes how much $x^\star(b)_i$ --- the $i$-th component of the optimal solution $x^\star(b)$ --- changes upon perturbation of $b_a$ --- the $a$-th component of the constraint vector $b$. We interpret $\frac{\partial x^\star(b)_i}{\partial b_a}$ as a measure of the correlation between variables $i$ and $a$ in the optimization problem, and we are interested in understanding how this quantity behaves as a function of the geodesic distance between $i$ and $a$. We motivate this terminology by establishing an analogy with the theory of correlations in probability, via the connection with Gaussian random variables (proofs are in Appendix \ref{app:Correlation}). 
We extend the notion of correlation beyond infinitesimal perturbations, and we show how our theory yields a first instance of comparison theorems for constrained optimization procedures, along the lines of the comparison theorems established in probability theory to capture stochastic decay of correlation and control the difference of high-dimensional distributions \cite{Dob70,RvH14}.\\
\indent In probability theory, decay of correlation characterizes the effective neighborhood dependency of random variables in a probabilistic network. Since the seminal work of Dobrushin \cite{Dob70}, this concept has found many applications beyond statistical physics. Recently, it has been used to develop and prove convergence guarantees for fast distributed local algorithms for inference and decision problems on large networks in a wide variety of domains, for instance, probabilistic marginal inference \cite{TJ02}, wireless communication \cite{W06}, network learning \cite{BMS08}, combinatorial optimization \cite{GGW14}, and nonlinear filtering \cite{RvH15}. However, even in applications where the underlying problem does not involve randomness, decay of correlation is typically established upon endowing the model with a probabilistic structure, hence modifying the original problem formulation. Our results in network optimization, instead, show how locality can be described in a purely non-random setting, as a structural property of the original optimization problem. To the best of our knowledge, non-random notions of correlation in optimization have been previously considered only in \cite{MVR10}, where the authors explicitly use the word ``correlation" to denote the sensitivity of optimal points with respect to localized perturbations, and they analyze the correlation as a function of the natural distance in the graph. However, in their work correlation is regarded as a tool to prove convergence guarantees for the specific algorithm at hand (Min-Sum message-passing to solve \emph{unconstrained} convex problems), and no general theory is built around it.\\
\indent\textbf{3) Locality: decay of correlation.} As a paradigm for network optimization, in Section \ref{sec:Optimal Network Flow} we consider the problem of computing network flows. This is a fundamental problem that has been studied in various formulations by different communities. The min-cost network flow variant (where the objective function is typically chosen to be linear, although there are non-linear extensions, and the constraints also include inequalities) has been essential in the development of the theory of polynomial-times algorithms for optimizations (see \cite{GSW12} and references therein, or \cite{Ahuja:MinCostFlow} for a book reference). In the case of quadratic functions, the problem is equivalent to computing electrical flows, which is a fundamental primitive related to Laplacian solvers that has been extensively studied in computer science (see \cite{citeulike:12634920}) with also applications to learning problems (\cite{Belkin2004, rasmussushant15}, for instance). In the formulation we consider, a directed graph $\vec G=(V,\vec E)$ is given (arrows indicate directed graphs/edges), with its structure encoded in the vertex-edge incidence matrix $A\in\mathbb{R}^{V\times \vec E}$. To each edge $e\in \vec E$ is associated a flow $x_e$ with a cost $f_e(x_e)$, and to each vertex $v\in V$ is associated an external flow $b_v$. The network flow problem consists in finding the flow $x^\star(b)\in\mathbb{R}^{\vec E}$ that minimizes the cost $f(x):=\sum_{e\in \vec E} f_e(x)$ and satisfies the conservation law $Ax=b$.\\
\noindent In this setting, the general sensitivity theory that we have previously developed allows to characterize the derivatives of the optimal flow in terms of graph Laplacians; in fact, in this case the matrix $A\Sigma(b)A^T$ corresponds to the Laplacian of an undirected weighted graph associated to $\vec G$, where each directed edge $e\in\vec{E}$ is given a weight $\Sigma(b)_{ee}>0$. Exploiting a general connection between the Moore-Penrose pseudoinverse of graph Laplacians and the Green's function of random walks on weighed graphs --- which we present as standalone in Appendix \ref{sec:Laplacians and random walks} --- we express the correlation term $\frac{d x^\star(b(\varepsilon))_e}{d\varepsilon}$ in terms of \emph{differences} of Green's functions.
Different graph topologies yield different decaying behaviors for this quantity, as a function of the geodesic distance $d(e,Z)$ between edge $e$ and the set of vertices $Z\subseteq V$ where the perturbation is localized, namely, $Z := \{z\in V : \frac{d b(\varepsilon)_z}{d \varepsilon} = 0\}$.
In the case of expanders, we derive spectral bounds that show an exponential decay, with rate given by the second largest eigenvalue in magnitude of the diffusion random walk. 
In this case we establish various types of decay of correlation bounds, point-to-set and set-to-point. Appendix \ref{app:Decay of correlation} contains the proofs of these results. For grid-like graphs, based on numerical simulations and on asymptotic results for the behavior of the Green's function in infinite grids \cite{lawler2010random}, we expect the correlation term to decay polynomially instead of exponentially.\\
\indent\textbf{4) Localized algorithms and bias-variance.} In Section \ref{sec:Scale-free algorithm} we investigate applications of the framework that we propose to develop scalable computationally-efficient algorithms. To illustrate the main principle behind our reasoning, we consider the case when the solution $x^\star(b)$ is given (also known as \emph{warm-start} scenario) and we want to compute the solution $x^\star(b+p)$ for the perturbed flow $b+p$, where $p$ is a perturbation supported on a small localized subset $Z\subseteq V$.
This setting belongs to the class of reoptimization problems typically studied in computer science and operations research (we could not to find previous literature on the setting that we consider; the case considered more frequently involves discrete problems and changes in the graph structure, as in \cite{reopt}).\\
\noindent The decay of correlation property structurally exhibited by the min-cost network flow problem encodes the fact that when the external flow $b$ is locally perturbed it suffices to recompute the solution only for the part of the network that is ``mostly affected" by this perturbation, i.e., the set of edges that have a distance at most $r$ from the perturbation set $Z$. Here, the radius $r$ is tuned to meet the desired level of error tolerance given the size of the perturbation and given the graph topology being investigated. This allows us to show that it is possible to develop localized versions of canonical optimization algorithms that can exploit locality by only updating the edges in a subgraph of $\vec G$.
We investigate regimes where localized algorithms can achieve computational savings against their global counterpart.
\noindent The key behind these savings is the bias-variance decomposition that we give for the error of localized algorithms. This decomposition conveys the idea that by introducing some bias in a given optimization routine (in our setting, by truncating the problem size restricting the algorithm to a subgraph) one can diminish its variance and obtain faster convergence rates. We illustrate this phenomenon theoretically and with numerical simulations for localized projected gradient descent. Proofs are in Appendix \ref{app:error localized}.
\begin{notation*}
\label{rem:notation}
For a given matrix $M$, let $M^T$ be the transpose, $M^{-1}$ be the inverse, and $M^{+}$ be the Moore-Penrose pseudoinverse. Let $\operatorname{Ker}(M):=\{x:Mx=0\}$ and $\operatorname{Im}(M):=\{y: y=Mx \text{ for some $x$}\}$ be the kernel and image of $M$, respectively. Given an index set $\mathcal{I}$ and subsets $K,L\subseteq\mathcal{I}$, if $M\in\mathbb{R}^{\mathcal{I}\times \mathcal{I}}$, let $M_{K,L}\in\mathbb{R}^{K\times L}$ denote the submatrix corresponding to the rows and columns of $M$ indexed by $K$ and $L$, respectively. Let $I$ be the identity matrix, $\mathbb{1}$ the all-one vector (or matrix), and $\mathbb{0}$ the all-zero vector (or matrix), whose sizes are implied by the context. Given a vector $x\in\mathbb{R}^\mathcal{I}$, let $x_i\in\mathbb{R}$ be the component associated to $i\in\mathcal{I}$, and $x_K:=(x_i)_{i\in K}\in\mathbb{R}^K$ the components associated to $K\subseteq \mathcal{I}$. Let $\| x\| := (\sum_{i\in \mathcal{I}} x^2_i)^{1/2}$ be the $\ell_2$-norm of $x$ and $\| x\|_K := (\sum_{i\in K} x^2_i)^{1/2}$ be the localized $\ell_2$-norm on $K$. Clearly, $\| x\|_K = \| x_K \|$ and $\| x\|_\mathcal{I} = \| x\|$. We use the notation $|K|$ to denote the cardinality of $K$.
If $\vec G = (V, \vec E)$ is a directed graph with vertex set $V$ and edge set $\vec E$, let $G = (V, E)$ be the undirected graph associated to $\vec G$, namely, $\{u,v\}\in E$ if and only if either $(u,v)\in \vec E$ or $(v,u)\in \vec E$.
\end{notation*}


\section{Sensitivity of optimal points}
\label{sec:Local sensitivity for equality-constrained convex minimization}

Let $\mathcal{V}$ be a finite set --- to be referred to as the ``variable set" --- and let $f:\mathbb{R}^\mathcal{V}\rightarrow \mathbb{R}$ be a strictly convex function, twice continuously differentiable. Let $\mathcal{F}$ be a finite set --- to be referred to as the ``factor set" --- and let $A\in\mathbb{R}^{\mathcal{F}\times \mathcal{V}}$.
Consider the following optimization problem over $x\in\mathbb{R}^\mathcal{V}$:
\begin{align}
	\begin{aligned}
		\text{minimize }\quad   & f(x)\\
		\text{subject to }\quad & Ax = b,
	\end{aligned}
	\label{original opt problem}
\end{align}
for $b\in\operatorname{Im}(A)\subseteq \mathbb{R}^\mathcal{F}$, so that the feasible region is not empty.
Throughout this paper we think of the function $f$ and the matrix $A$ as fixed, and we consider the solution of the optimization problem above as a function of the vector $b\in\operatorname{Im}(A)$.
By strict convexity, this problem has a unique optimal solution, that we denote by
$
	x^\star(b) := {\operatorname{argmin}}\{ f(x) : x\in\mathbb{R}^\mathcal{V},A x = b \}.
$

Theorem \ref{thm:comparisontheorem} below provides a characterization of the way a perturbation of the constraint vector $b$ on the subspace $\operatorname{Im}(A)$ affects the optimal solution $x^\star(b)$ in the case when the function $f$ is strongly convex.
The proof is given in Appendix \ref{sec:Hadamard global inverse theorem}.

\begin{theorem}[Sensitivity optimal point]\label{thm:comparisontheorem}
Let $f:\mathbb{R}^\mathcal{V}\rightarrow \mathbb{R}$ be a strongly convex function, twice continuously differentiable.
Let $A\in\mathbb{R}^{\mathcal{F}\times \mathcal{V}}$.
For $b\in \operatorname{Im}(A)$, let $\Sigma(b):=\nabla^2 f(x^\star (b))^{-1}$ and
$
	D(b)
	:=
	\Sigma(b)A^T (A\Sigma(b)A^T)^{+}.
$
Then, $x^\star$ is continuously differentiable along the subspace $\operatorname{Im}(A)$, and given a differentiable function $\varepsilon\in\mathbb{R} \rightarrow b(\varepsilon)\in \operatorname{Im}(A)$ we have
$$
	\frac{d x^\star(b(\varepsilon))}{d \varepsilon} 
	= D(b(\varepsilon)) \frac{d b(\varepsilon)}{d \varepsilon}.
$$
\end{theorem}

Most of the literature on the sensitivity of optimal points for nonlinear programs investigates \emph{local} results for infinitesimal perturbations, establishing the existence of a local neighborhood of the perturbed parameter(s) where the optimal point(s) has(have) some analytical properties such as continuity, differentiability, etc. Classical references are \cite{guddat76,Robinson1979,fiacco83,Kyparisis1986}. Typically, the main tool used to establish this type of results is the implicit function theorem applied to the first order optimality conditions (KKT map), which is a local statement (as such, it is flexible and it can be applied to a great variety of optimization problems). The sensitivity result that we present in Theorem \ref{thm:comparisontheorem} in the case of strongly convex functions, on the other hand, is based on Hadamard's global inverse function theorem \cite{krantz2002implicit}. This is a \emph{global} statement, as it shows that the optimal point $x^\star$ is continuously differentiable along the \emph{entire} subspace $\operatorname{Im}(A)$. This fact allows the use of the fundamental theorem of calculus to deal with finite perturbations, which is instrumental for the results developed in this paper (in particular, for the connection with comparison theorems in probability, Section \ref{sec:Comparison theorems}, and for the results in Section \ref{sec:Scale-free algorithm}). In the case of quadratic programs with linear constraints (the problem we consider can be thought of as an extension of this setting to strongly convex functions) there are many papers in parametric programming investigating the sensitivity of optimal points with respect to changes in the objective functions and constraints (see \cite{Phu2001} and references therein). While most of these results are again local, some of them are closer to our approach and also address finite perturbations \cite{boot63}.

Theorem \ref{thm:comparisontheorem} characterizes the behavior of  $x^\star(b)$ upon perturbations of $b$ along $\operatorname{Im}(A)\subseteq \mathbb{R}^\mathcal{F}$. If the matrix $A$ is full row rank, i.e., $\operatorname{Im}(A) = \mathbb{R}^\mathcal{F}$, then the optimal point $x^\star$ is everywhere continuously differentiable, and we can compute its gradient. We have the following immediate corollary.

\begin{corollary}[Sensitivity optimal point, full rank case]
\label{cor:fullrank}
Consider the setting of Theorem \ref{thm:comparisontheorem}, with the matrix $A\in\mathbb{R}^{\mathcal{F}\times \mathcal{V}}$ having full row rank, i.e., $\operatorname{Im}(A) = \mathbb{R}^\mathcal{F}$. 
Then, the function $b\in \mathbb{R}^\mathcal{F}\rightarrow x^\star(b)\in\mathbb{R}^{\mathcal{V}}$ is continuously differentiable and 
$$
	\frac{d x^\star(b)}{d b}
	= D(b)
	=
	\Sigma(b)A^T (A\Sigma(b)A^T)^{-1}.
$$
\end{corollary}

\begin{remark}
\label{rem:conditions}Our goal is to present the simplest setting of interest where we can establish locality and illustrate the computational savings achieved by localized algorithms. The computational advantages provided by localization are already appreciable in the well-conditioned setting that we consider, as we explain in Section \ref{sec:Scale-free algorithm} below. For this reason, we are satisfied with the assumption of strong convexity in Theorem \ref{thm:comparisontheorem}. The main tool behind Theorem \ref{thm:comparisontheorem} is Hadamard's global inverse function theorem, which provides \emph{necessary and sufficient} conditions for the inverse of a continuously differentiable function to be continuously differentiable. These conditions involve cohercitivity of the function of interest, along with a non-degeneracy condition on the determinant of the function. To relax the assumptions that we give in Theorem \ref{thm:comparisontheorem}, one needs to consider a more refined application of Hadamard's theorem to the KKT map.
\end{remark}


\section{Notions of correlation in optimization}
\label{sec:Notions of correlation in optimization}

The sensitivity analysis presented in Section \ref{sec:Local sensitivity for equality-constrained convex minimization} suggests a natural notion of correlation between variables and factors in optimization, resembling notions of correlation among random variables in probability theory.
If the matrix $A$ is full row rank, then the quantity $\frac{\partial x^\star(b)_i}{\partial b_a}$ is well-defined and it captures the interaction between variable $i\in\mathcal{V}$ and factor $a\in\mathcal{F}$ in the optimization procedure, and the quantity $D(b)_{ia}$ in Corollary \ref{cor:fullrank} characterizes this correlation as a function of the constraint matrix $A$, the objective function $f$, and the optimal solution $x^\star (b)$. Theorem \ref{thm:comparisontheorem} allows us to extend the notion of correlation between variables and factors to the more general case when the matrix $A$ is not full rank. As an example, let $b,p\in \operatorname{Im}(A)$, and assume that $p$ is supported on a subset $Z\subseteq \mathcal{F}$, namely, $p_a\neq 0$ if and only if $a\in Z$. Define $b(\varepsilon) := b+\varepsilon p$. Then, the quantity $\frac{d x^\star(b(\varepsilon))_i}{d\varepsilon}$ measures how much a perturbation of the constraints in $Z$ affects the optimal solution at $i\in \mathcal{V}$, hence it can be interpreted as a measure of the correlation between variable $i$ and the factors in $Z$, which is characterized by the term $(D(b(\varepsilon)) \frac{d b(\varepsilon)}{d \varepsilon})_{i}=\sum_{a\in Z} D(b(\varepsilon))_{ia} p_a$ in Theorem \ref{thm:comparisontheorem}.

We now discuss how these notions of correlation relate to analogous notions in probability.

\subsection{Connection with Gaussian random variables}

The resemblance between the notion of correlations in optimization and in probability is made explicit via the analogy to the theory of Gaussian random variables. Recall the following result (proofs of the results here discussed are given for completeness in Appendix \ref{app:Correlation}).

\begin{proposition}[Conditional mean of Gaussian random variables]
\label{prop:condsgaussian}
Let $\mathcal{V},\mathcal{F}$ be two finite sets. Let $X\in\mathbb{R}^\mathcal{V}$ be a Gaussian random vector with mean $\mu\in\mathbb{R}^\mathcal{V}$ and covariance $\Sigma\in\mathbb{R}^{\mathcal{V}\times\mathcal{V}}$, possibly singular. Let $A\in\mathbb{R}^{\mathcal{F}\times\mathcal{V}}$ be given.
Given a differentiable function $\varepsilon\in\mathbb{R} \rightarrow b(\varepsilon)\in \operatorname{Im}(A)$, we have
$
	\frac{d \mathbf{E}[X | AX=b(\varepsilon)]}{d\varepsilon} 
	= \Sigma A^T (A\Sigma A^T)^{+}\frac{d b(\varepsilon)}{d\varepsilon}.
$
If $\Sigma$ is invertible and $A$ is full row rank, then for each $b\in\mathbb{R}^\mathcal{F}$ we have
$
	\frac{d \mathbf{E}[X | AX=b]}{d b} 
	= \Sigma A^T (A\Sigma A^T)^{-1}.
	\label{Gaussian:correlation}
$
\end{proposition}

Proposition \ref{prop:condsgaussian} shows that for $i\in\mathcal{V},a\in\mathcal{F}$, the quantity
$
	\frac{\partial \mathbf{E}[X_i | AX=b]}{\partial b_a} 
	= (\Sigma A^T (A\Sigma A^T)^{-1})_{ia}
$
can be interpreted as a measure of correlation between the random variables $X_i$ and $(AX)_a$, as it describes how much a perturbation of $(AX)_a$ impacts $X_i$, upon conditioning on $AX$.
A similar interpretation can be given in optimization for the quantities in Corollary \ref{cor:fullrank}, with the difference that typically these quantities depend on $b\in\operatorname{Im}(A)$, as they are functions of $x^\star(b)$.

The sensitivity results that we derived in Section \ref{sec:Local sensitivity for equality-constrained convex minimization} also yield notions of correlation in optimization between variables. 
The following lemma, an immediate application of Corollary \ref{cor:fullrank}, shows that the local behavior of the optimal solution of the optimization problem \eqref{original opt problem} when we freeze some coordinates, upon perturbation of these coordinates, is analogous to the behavior of the conditional mean of a non-degenerate Gaussian random vector upon changing the coordinates we condition on.
Recall (see the proof of Proposition \ref{prop:condsgaussian}) that if $X\in\mathbb{R}^\mathcal{V}$ is a Gaussian vector with mean $\mu\in\mathbb{R}^\mathcal{V}$ and positive definite covariance $\Sigma\in\mathbb{R}^{\mathcal{V}\times\mathcal{V}}$, then for $I\subseteq\mathcal{V}$ and $B:=\mathcal{V}\setminus I$
$
	\frac{d\mathbf{E}[X_I | X_B=x_B]}{dx_B} = \Sigma_{I,B} (\Sigma_{B,B})^{-1}.
$

\begin{lemma}[Sensitivity with respect to boundary conditions]
\label{lem:Gaussian}
Let $f:\mathbb{R}^\mathcal{V}\rightarrow \mathbb{R}$ be a strongly convex function, twice continuously differentiable. Let $I\subseteq\mathcal{V}$ be a nonempty set, and let $B:=\mathcal{V}\setminus I$ not empty.
Define the function
$
	x^\star_I:x_B\in \mathbb{R}^B\longrightarrow
	x^\star_I(x_B) := {\operatorname{argmin}}\left\{ f(x_Ix_B) : x_I\in\mathbb{R}^I \right\}.
$
For $x_B\in \mathbb{R}^B$, let $H(x_B):=\nabla^2 f(x^\star_I(x_B)x_B)$ and $\Sigma(x_B):=H(x_B)^{-1}$.
Then, $x^\star_I$ is continuously differentiable and $\frac{d x^\star_I(x_B)}{dx_B}=
	\Sigma(x_B)_{I,B} (\Sigma(x_B)_{B,B})^{-1}
	= - (H(x_B)_{I,I})^{-1} H(x_B)_{I,B}.
$
\end{lemma}

\subsection{Comparison theorems}\label{sec:Comparison theorems}
The connection between Theorem \ref{thm:comparisontheorem} and the theory of correlations in probability extends beyond infinitesimal perturbations. 
As previously discussed, Theorem \ref{thm:comparisontheorem} can be used to deal with finite perturbations, and so it can be interpreted as a comparison theorem to capture \emph{uniform} correlations in optimization, along the lines of the comparison theorems in probability theory to capture stochastic decay of correlation and control the difference of high-dimensional distributions (see the seminal work in \cite{Dob70}, and \cite{RvH14} for generalizations).

To see this analogy, let us consider a simplified version of the Dobrushin comparison theorem that can be easily derived from the textbook version in \cite{Geo11}, Theorem 
8.20. Let $I$ be a finite set, and let $\Omega:=\prod_{i\in I}\Omega_i$ 
where $\Omega_i$ is a finite set for each $i\in I$.  Define the 
projections $X_i:x\mapsto x_i$ for $x\in\Omega$ an $i\in I$.
For any probability distribution $\mu$ on $\Omega$, define the marginal
$
	\mu_i(y) := \mu(X_i= y),
$
and the conditional distribution
$
	\mu^x_i(y) := \mu(X_i= y|X_{I\backslash\{i\}}=x_{I\backslash\{i\}}).
$
Define the total variation distance between two distributions $\nu$ and $\tilde \nu$ on $\Omega_i$ as $\| \nu - \tilde \nu \|_{T} := \frac{1}{2} \sum_{y\in \Omega_i} |\nu(y) - \tilde \nu(y)|$.

\begin{theorem}[Dobrushin comparison theorem]
\label{thm:dobrushin}
Let $\mu,\tilde\mu$ be probability distributions on 
$\Omega$. For each $i,j\in I$, define
$
	\mathsf{C}_{ij} :=
		\!\!\sup_{x,z\in\Omega:x^{I\backslash\{j\}}=
		z^{I\backslash\{j\}}} 
		\! \|\mu^x_i-\mu^z_i\|_{T}
$
and
$
		\mathsf{b}_j := \sup_{x\in\Omega}\|\mu^x_j-\tilde\mu^x_j\|_{T},
$
and assume that the Dobrushin condition holds:
$
	\max_{i\in I}\sum_{j\in I}\mathsf{C}_{ij}<1.
$
Then the matrix sum
$\mathsf{D} := \sum_{t\ge 0}\mathsf{C}^t$ is convergent, and for any $i\in I$, $y\in\Omega_i$, we have
$
	\|\mu_i-\tilde\mu_i \|_{T} \le
	\sum_{j\in I} \mathsf{D}_{ij} 
	\mathsf{b}_j.
$
\end{theorem}
The Dobrushin coefficient $\mathsf{C}_{ij}$ is a (uniform) measure of 
the degree to which a perturbation of site $j$ \emph{directly} affects site $i$ 
under the distribution $\mu$.
However, perturbing site $j$ might also 
indirectly affect site $i$: it could affect another site $k$ which in turn 
affects $i$, etc.  The \emph{aggregate} effect of a perturbation of site $j$ on 
site $i$ is captured by $\mathsf{D}_{ij}$.
The quantity $\mathsf{b}_j$ is a comparison term that
measures the local difference at site $j$ between $\mu$ and $\tilde\mu$ (in 
terms of the conditional distributions $\mu_j^\cdot$ and 
$\tilde\mu_j^\cdot$).

The formal analogy between the Dobrushin comparison theorem and the sensitivity results of Theorem \ref{thm:comparisontheorem} for the optimization problem \eqref{original opt problem} is made explicit by the fundamental theorem of calculus. This connection is easier to make if we assume that the matrix $A$ has full row rank, and we consider the results in Corollary \ref{cor:fullrank}. In this setting, the optimal point $x^\star$ is everywhere continuously differentiable, and for each $i\in \mathcal{V}$, $b,\tilde b\in\mathbb{R}^\mathcal{F}$, $b \neq \tilde b$, we have $x^\star(b)_i
	-
	x^\star(\tilde b)_i=
	\int_{0}^1 \frac{d x^\star(\theta b + (1-\theta)\tilde b)_i}{d\theta} d\theta
	=
	\sum_{a\in \mathcal{F}}
	D(b,\tilde b)_{ia}
	(b_a-\tilde b_a),
$
where $D(b,\tilde b)_{ia}:=\int_{0}^1 (\Sigma(b_\theta)A^T (A\Sigma(b_\theta)A^T)^{-1})_{ia} d\theta$ with $b_\theta := \theta b + (1-\theta)\tilde b$. If for each $i\in\mathcal{V}$ and $a\in\mathcal{F}$ we have $\sup_{b\in\mathbb{R}^\mathcal{F}} |(\Sigma(b)A^T (A\Sigma(b)A^T)^{-1})_{ia}| \le D_{ia}$, then from the previous expression we find
$
	| x^\star(b)_i
	-
	x^\star(\tilde b)_i |
	\le
	\sum_{a\in \mathcal{F}}
	D_{ia}
	|b_a-\tilde b_a|,
$
whose structure resembles the statement in Theorem \ref{thm:dobrushin}. The quantity $D_{ia}$ represents a uniform measure of the aggregate impact that a perturbation of the $a$-th component of the constraint vector $b$ has to the $i$-th component of the optimal solution $x^\star$, so the matrix $D$ takes the analogous role of the matrix $\mathsf{D}$ in Theorem \ref{thm:dobrushin} (and as we will see below in a concrete application, see Theorem \ref{thm:comparisontheoremnetworkflow}, suitable series expansions of $D$ yield the analogous of $\mathsf{C}$).
The quantity $|b_a-\tilde b_a|$ is a comparison term that measures the local difference at factor $a$ between $b$ and $\tilde b$, resembling the role of $\mathsf{b}_j$ in Theorem \ref{thm:dobrushin}. 

In the next section we investigate the notion of correlation just introduced in the context of network optimization, in a concrete instance when the constraints naturally reflect a graph structure, and we investigate the behavior of the correlations as a function of the natural distance in the graph.


\section{Locality: decay of correlation}\label{sec:Optimal Network Flow}

As a paradigm for network optimization, we consider the network flow problem that has been widely studied in various fields (see introduction).
Consider a directed graph $\vec{G}:=(V,\vec{E})$,
with vertex set $V$ and edge set $\vec{E}$, with no self-edges and no multiple edges. Let $G=(V,E)$ be the undirected graph naturally associated with $\vec{G}$, that is, $\{u,v\}\in E$ if and only if either $(u,v)\in \vec{E}$ or $(v,u)\in \vec{E}$. Without loss of generality, assume that $G$ is connected (otherwise we can treat each connected component on its own). For each $e\in \vec{E}$ let $x_e$ denote the flow on edge $e$, with $x_e>0$ if the flow is in the direction of the edge, $x_e<0$ if the flow is in the direction opposite the edge. For each $v\in V$ let $b_v$ be a given external flow on the vertex $v$: $b_v>0$ represents a source where the flow enters the vertex, whereas $b_v<0$ represents a sink where the flow leaves the vertex. Assume that the total of the source flows equals the total of the sink flows, that is, $\mathbb{1}^Tb = \sum_{v\in V} b_v = 0$, where $b=(b_v)_{v\in V}\in\mathbb{R}^V$ is the flow vector. We assume that the flow satisfies a conservation equation so that at each vertex the total flow is zero. This conservation law can be expressed as $A x = b$, where $A\in\mathbb{R}^{V\times \vec{E}}$ is the \emph{vertex-edge incidence matrix} defined as $A_{ve}:= 1$ if $e$ leaves node $v$, $A_{ve}:= -1$ if $e$ enters node $v$, and $A_{ve}:= 0$ otherwise.

\noindent For each edge $e\in \vec{E}$ let $f_e:\mathbb{R}\rightarrow\mathbb{R}$ be its associated cost function, assumed to be strongly convex and twice continuously differentiable.
The network flow problem reads as problem \eqref{original opt problem} with $f(x) := \sum_{e\in \vec{E}} f_e(x_e)$.
It is easy to see that since $G$ is connected $\operatorname{Im}(A)$ consists of all vectors orthogonal to the vector $\mathbb{1}$, i.e., $\operatorname{Im}(A) = \{ y \in \mathbb{R}^V: \mathbb{1}^T y = 0 \}$.
Henceforth, for each $b\in\mathbb{R}^V$ with $\mathbb{1}^Tb=0$, let $x^\star(b)$ be the optimal flow.

We first apply the sensitivity theory developed in Section \ref{sec:Local sensitivity for equality-constrained convex minimization} to characterize the correlation between vertices (i.e., factors) and edges (i.e., variables) in the network flow problem. Then, we investigate the behavior of these correlations in terms of the natural distance on the graph $G$.

\subsection{Correlation, graph Laplacians and Green's functions}
\label{sec:Correlation in terms of graph Laplacians}
In the setting of the network flow problem, Theorem \ref{thm:comparisontheorem} immediately allows us to characterize the derivatives of the optimal point $x^\star$ along the subspace $\operatorname{Im}(A)$ as a function of the graph Laplacian \cite{Chung:1997}.
For $b\in\mathbb{R}^V$ such that $\mathbb{1}^Tb=0$, let $\Sigma(b) := \nabla^2 f(x^\star (b))^{-1}\in\mathbb{R}^{\vec{E}\times \vec{E}}$, which is a diagonal matrix with entries given by
$
	\sigma(b)_{e} := \Sigma(b)_{ee}:=(\frac{\partial^2 f_e(x^\star(b)_e)}{\partial x_e^2})^{-1} > 0.
$
Each term $\sigma(b)_{e}$ is strictly positive as $f_e$ is strongly convex by assumption.
Let $W(b)\in\mathbb{R}^{V\times V}$ be the symmetric matrix defined, for each $u,v\in V$, as $W(b)_{uv}:=\sigma(b)_{e}$ if $e=(u,v) \in \vec{E} \text{ or } e=(v,u) \in \vec{E}$, and $W(b)_{uv}:=0$ otherwise.
Let $D(b)\in\mathbb{R}^{V\times V}$ be the diagonal matrix with entries given by
$
	d(b)_{v} := D(b)_{vv} := \sum_{u\in V} W(b)_{vu},
$
for $v\in V$.
Let $L(b):= D(b)-W(b)$ be the graph Laplacian of the undirected weighted graph $(V,E,W(b))$, where to each edge $e=\{u,v\}\in E$ is associated the weight $W(b)_{uv}$.
A direct application of Theorem \ref{thm:comparisontheorem} (upon choosing variable set $\mathcal{V} := \vec{E}$ and factor set $\mathcal{F} := V$, and noticing that $A\Sigma(b)A^T = L(b)$) shows that the derivatives of the optimal point $x^\star$ along $\operatorname{Im}(A)$ can be expressed in terms of the Moore-Penrose pseudoinverse of $L(b)$. The connection between $L(b)^+$ and the Green's function of random walks with transition matrix $P(b):=D(b)^{-1}W(b)$ allows us to derive the following result (proofs are in Appendix \ref{sec:Laplacians and random walks}). 

\begin{theorem}[Sensitivity optimal flow]\label{thm:comparisontheoremnetworkflow}
Given $b\in\mathbb{R}^V$ with $\mathbb{1}^Tb=0$, let
$
	D(b)
	:=
	\Sigma(b)A^T L(b)^{+}.
$
The optimal network flow $x^\star$ is continuously differentiable along $\operatorname{Im}(A)$, and given a differentiable function $\varepsilon\in\mathbb{R} \rightarrow b(\varepsilon)\in \operatorname{Im}(A)$ we have
$
	\frac{d x^\star(b(\varepsilon))}{d \varepsilon} 
	= D(b(\varepsilon)) \frac{d b(\varepsilon)}{d \varepsilon}.
$
For $e=(u,v)\in\vec{E}$, we have
$$
	\frac{d x^\star(b(\varepsilon))_e}{d \varepsilon} 
	= \sigma(b)_e \sum_{z\in V} 
	\frac{1}{d(b)_z} \frac{d b(\varepsilon)_z}{d \varepsilon}
	\sum_{t=0}^\infty (P(b)^t_{uz} - P(b)^t_{vz}).
$$
\end{theorem}

Let $b,p\in\mathbb{R}^V$ such that $\mathbb{1}^Tb=\mathbb{1}^Tp =0$, and assume that $p$ is supported on a subset $Z\subseteq V$, namely, $p_v\neq 0$ if and only if $v\in Z$. Define $b(\varepsilon) := b+\varepsilon p$. Then, as discussed in Section \ref{sec:Notions of correlation in optimization}, the quantity $\frac{d x^\star(b(\varepsilon))_e}{d\varepsilon}$ can be interpreted as a measure of the correlation between edge $e\in \vec{E}$ and the vertices in $Z$ in the network flow problem. How does the correlation behave with respect to the graph distance between $e$ and $Z$? Theorem \ref{thm:comparisontheoremnetworkflow} shows that the correlation is controlled by the \emph{difference} of the Green's function $\sum_{t=0}^\infty P(b)^t_{uz}$ with respect to two neighboring starting points $u$ and $v$ (note that the Green's function itself is infinite, as we are dealing with finite graphs). Different graph topologies yield different decaying behaviors for this quantity.
In the case of expanders, we now derive spectral bounds that decay exponentially, with rate given by the second largest eigenvalue in magnitude of the diffusion random walk. For grid-like topologies, based on simulations and on asymptotic results for the behavior of the Green's function in infinite grids \cite{lawler2010random}, we expect the correlation to decay polynomially rather than exponentially.

\subsection{Decay of correlation for expanders}\label{sec:Decay of correlation for expanders}
Let $n:=|V|$ be the cardinality of $V$, and for each $b\in\operatorname{Im}(A)$ let $-1\le\lambda_n(b) \le \lambda_{n-1}(b) \le \cdots \le \lambda_2(b) < \lambda_1(b) =1$ be the real eigenvalues of $P(b)$.\footnote{This characterization of eigenvalues for random walks on connected weighted graphs follows from the Perron-Frobenius theory. See \cite{lovasz1993random}.} Define $\lambda(b):=\max\{|\lambda_2(b)|, |\lambda_n(b)|\}$ and $\lambda:=\sup_{b\in\operatorname{Im}(A)} \lambda(b)$.
For each $v\in V$, let $\mathcal{N}(v):=\{w\in V: \{v,w\}\in E\}$ be the set of node neighbors of $v$ in the graph $G$. Let $d$ be the graph-theoretical distance between vertices in the graph $G$, namely, $d(u,v)$ is the length of the shortest path between vertices $u,v\in V$. For subset of vertices $U,Z\subseteq V$, define $d(U,Z):=\min \{d(u,z):u\in U, z\in Z\}$. For each subset of edges $\vec F\subseteq \vec E$, let $V_{\vec F}\subseteq V$ be the vertex set of the subgraph $(V_{\vec F},\vec F)$ of $\vec G$ that is induced by the edges in $\vec F$. 
Recall the definition of the localized $\ell_2$-norm from Section \ref{sec:introduction}.
The following result attests that the correlation for the network flow problem is upper-bounded by a quantity that decays exponentially as a function of the distance in the graph, with rate given by $\lambda$. For graphs where $\lambda$ does not depend on the dimension, i.e., expanders, Theorem \ref{thm:Decay of correlation} can be interpreted as a first manifestation of the decay of correlation principle (i.e., locality) in network optimization. The proof is given in Appendix \ref{app:Decay of correlation}.  

\begin{theorem}[Decay of correlation for expanders]
\label{thm:Decay of correlation}
Consider the setting defined above. Let $\varepsilon\in\mathbb{R} \rightarrow b(\varepsilon)\in \operatorname{Im}(A)$ be a differentiable function such that for any $\varepsilon\in\mathbb{R}$ we have $\frac{d b(\varepsilon)_v}{d \varepsilon} \neq 0$ if and only if $v\in Z$, for a given $Z\subseteq V$. Then, for any subset of edges $\vec F \subseteq \vec E$ and any $\varepsilon\in\mathbb{R}$, we have
$$
	\left\| \frac{d x^\star(b(\varepsilon))}{d \varepsilon} \right\|_{\vec{F}}
	\le c\,
	\frac{\lambda^{d(V_{\vec F},Z)}}{1-\lambda}
	\left\|\frac{d b(\varepsilon)}{d \varepsilon}\right\|_Z,
$$
with $c:= \displaystyle\sup_{b\in\operatorname{Im}(A)}\frac{\max_{v\in V_{\vec F}} \sqrt{2 |\mathcal{N}(v) \cap V_{\vec F}|}}{\min_{v\in V_{\vec F}} d(b)_v} \max_{u,v\in V_{\vec F}} W(b)_{uv}$.
\end{theorem}

Recall that $\| \frac{d x^\star(b(\varepsilon))}{d \varepsilon} \|_{\vec{F}} \equiv (\sum_{e\in\vec{F}} (\frac{d x^\star(b(\varepsilon))_e}{d \varepsilon})^2)^{1/2}$ and $\|\frac{d b(\varepsilon)}{d \varepsilon}\|_Z \equiv (\sum_{v\in Z} (\frac{d b(\varepsilon)_v}{d \varepsilon})^2)^{1/2}$. The bound in Theorem \ref{thm:Decay of correlation} controls the effect that a localized perturbation supported on a subset of vertices $Z\subseteq V$ has on a subset of edges $\vec{F}\subseteq \vec{E}$, as a function of the distance between $\vec{F}$ and $Z$, i.e., $d(V_{\vec F},Z)$ (note that we only defined the distance among vertices, not edges, and that this distance is with respect to the unweighted graph $G$).
A key feature of Theorem \ref{thm:Decay of correlation} --- which is essential for the results in Section \ref{sec:Scale-free algorithm} below --- is that the bound presented does not depend on the cardinality of $\vec{F}$.

Theorem \ref{thm:Decay of correlation} controls the effect that a \emph{single} localized perturbation (supported on multiple vertices, as it has to be that $|Z|\ge 2$ for the function $\varepsilon\in\mathbb{R} \rightarrow b(\varepsilon)$ to be on $\operatorname{Im}(A)$) has on a \emph{collection} of edges for the optimal solution, independently of the number of edges being considered. We refer to this type of decay of correlation as \emph{set-to-point}. Analogously, 
it is possible to control the effect that \emph{multiple} localized perturbations have on a \emph{single} edge for the optimal solution, independently of the number of perturbations being considered. We refer to this type of decay of correlation as \emph{point-to-set}.
To illustrate in more detail these two types of decay of correlation, and for the sake of simplicity, we consider perturbations that are supported on exactly two vertices, corresponding to the endpoints of edges. Given $b\in\mathbb{R}^V$ such that $\mathbb{1}^Tb=0$, and $e=(u,v)\in\vec E$, we define the directional derivative of $x^\star$ along edge $e$ evaluated at $b$ as
$
	\nabla_e x^\star(b) 
	:= \frac{d x^\star(b+\varepsilon (e_u-e_v))}{d \varepsilon}|_{\varepsilon =0},
$
where for each $v\in V$, $e_v\in\mathbb{R}^V$ is the vector defined as $(e_v)_w=0$ if $w\neq v$ and $(e_v)_v=1$.
Then, we immediate have the following corollary of Theorem \ref{thm:Decay of correlation}.
\begin{corollary}[Set-to-point decay of correlation]
\label{cor:Set-to-point decay of correlation}
For $b\in\mathbb{R}^V$ with $\mathbb{1}^Tb=0$, $\vec F \subseteq \vec E$ and $e\in \vec E$, we have
$$
	\| \nabla_{e} x^\star(b) \|_{\vec{F}}
	\equiv
	\sqrt{\sum_{f\in\vec{F}} (\nabla_{e} x^\star(b)_f)^2}
	\le \sqrt{2} c\,
	\frac{\lambda^{d(V_{\vec F},V_{\{ e\}})}}{1-\lambda},
$$
where $c$ is as in Theorem \ref{thm:Decay of correlation}.
\end{corollary}

The key feature of the bound in Corollary \ref{cor:Set-to-point decay of correlation} is that it does not depend on the cardinality of $\vec F$. Exploiting the symmetry of the identities involving the graph Laplacian, it is also easy to establish the following analogous result. The proof is given in Appendix \ref{app:Decay of correlation}.

\begin{lemma}[Point-to-set decay of correlation]
\label{lem:Point-to-set decay of correlation}
For $b\in\mathbb{R}^V$ with $\mathbb{1}^Tb=0$, $\vec F \subseteq \vec E$ and $f\in \vec E$, we have
$$
	\sqrt{\sum_{e\in\vec{F}} (\nabla_{e} x^\star(b)_f)^2}
	\le \sqrt{2} c'\,
	\frac{\lambda^{d(V_{\vec F},V_{\{ f\}})}}{1-\lambda},
$$
with $c'=\sup_{b\in\operatorname{Im}(A)}\!\!\frac{W(b)_{wz}\max_{v\in V_{\vec F}} \sqrt{2 |\mathcal{N}(v) \cap V_{\vec F}|}}{\sqrt{\min\{d(b)_w,d(b)_z\}}\min_{v\in V_{\vec F}} \sqrt{d(b)_v}}$.
\end{lemma}


\section{Localized algorithms and bias-variance}\label{sec:Scale-free algorithm}
Let us consider the network flow problem defined in the previous section, for a certain external flow $b\in \mathbb{R}^V$ such that $\mathbb{1}^Tb =0$. Let $Z\subseteq V$ and choose $p\in\mathbb{R}^V$ supported on $Z$ (i.e., $p_v\neq 0$ if and only if $v\in Z$) with $\mathbb{1}^Tp =0$. Assume that we perturb the external flow $b$ by adding $p$.
We want to address the following question: given knowledge of the solution $x^\star(b)$ for the unperturbed problem, what is a computationally efficient algorithm to compute the solution $x^\star(b+p)$ of the perturbed problem?
The main idea that we want to exploit is that when locality holds and we can prove a decay of correlation property such as the one established in Theorem \ref{thm:Decay of correlation} for expander graphs,
then a localized perturbation of the external flow will affect more the components of $x^\star(b)$ that are close to the perturbed sites on $Z$. Hence, we expect that only a subset of the components of the solution around $Z$ needs to be updated to meet a prescribed level of error tolerance, yielding savings on the computational complexity.

\subsection{Local problem and localized algorithms}

To formalize the argument given above, henceforth let $\vec G'=(V',\vec E')$ be a subgraph of $\vec G=(V,\vec E)$ such that $Z\subseteq V'$. Let $G'=(V',E')$ be the undirected graph associated to $\vec G'$ (see notation in Section \ref{sec:introduction}), and assume that $G'$ is connected. Define $V'^C:=V\setminus V'$ and $\vec E'^C:=\vec E\setminus \vec E'$. 
Let us consider the localized version of the network flow problem supported on the subgraph $\vec G'$. Let $A':=A_{V',\vec E'}$ be the submatrix of $A$ corresponding to the rows indexed by $V'$ and the columns indexed by $\vec E'$. For $b'\in\operatorname{Im}(A')\subseteq\mathbb{R}^{V'}$, let $x'^\star(b')\in\mathbb{R}^{\vec E'}$ denote the solution of the following problem over $x'\in\mathbb{R}^{\vec E'}$:
\begin{align}
\begin{aligned}
	\text{minimize }\quad   & f'(x'):=\sum_{e\in \vec E'} f_e(x'_e)\\
	\text{subject to }\quad & A' x' = b'.
\end{aligned}
\label{def:problem_localized}
\end{align} 
If locality holds and the subgraph $\vec G'$ is large enough, we expect that the solution of the perturbed problem $x^\star(b+p)$ will be close to the solution of the unperturbed problem $x^\star(b)$ on $\vec E'^C$, and it will be substantially different on $\vec E'$. For this reason we investigate the performance of local iterative algorithms that operate only on the subgraph $\vec G'$, leaving the components supported on $\vec E'^C$ unchanged.

\begin{definition}[Local algorithm]
\label{def:localalgorithm}
Given $b\in\mathbb{R}^V$ with $\mathbb{1}^Tb =0$, let $\mathcal{X}'_b:=\{u\in \mathbb{R}^{\vec E}: (Au)_{V'^C}=b_{V'^C}\}$. A map $T'_b:  \mathcal{X}'_b \rightarrow \mathcal{X}'_b$ defines a \emph{local algorithm} on the subgraph $\vec G'=(V',\vec E')$ if the following two conditions hold for any choice of $x\in \mathcal{X}'_b$:
\begin{enumerate}[(i)]
\item 
$
	 \lim_{t\rightarrow\infty} T'^t_{b}(x)_{E'} = x'^\star(b')$, $b'=b_{V'}  - A_{V',\vec E'^C}x_{\vec E'^C};
$
\item $T'_b(x)_{\vec E'^C}=x_{\vec E'^C}$.
\end{enumerate}
\end{definition}

This definition ensures that a local algorithm $T'_{b}$ only updates the components of $x\in\mathcal{X}'_b$ supported on $\vec E'$ and there converges to the solution of problem \eqref{def:problem_localized} with $b'=b_{V'}  - A_{V',\vec E'^C}x_{\vec E'^C}$. The components that are left invariant on $\vec E'^C$ play the role of boundary conditions.
The algorithm that we propose to compute $x^\star(b+p)$ given knowledge of $x^\star(b)$ amounts to running a local algorithm on $\vec G'$ for $t$ iterations with boundary conditions $x^\star(b)_{\vec E'^C}$,
i.e.,
$
	T'^t_{b+p}(x^\star(b)).
$
By definition
\begin{align*}
	&\lim_{t\rightarrow\infty} T'^t_{b+p}(x^\star(b))_{e} =
	\begin{cases}
	x'^\star(b_{V'} \!+\! p_{V'} \!-\! A_{V',\vec E'^C} x^\star(b)_{\vec E'^C})_e &\text{if } e\in \vec E',\\
	x^\star(b)_e &\text{if } e\in \vec E'^C.
	\end{cases}
\end{align*}
Designing a local algorithm to compute $x^\star(b+p)$ is easy. Given any algorithmic procedure (for instance, a first-order method, a second-order method, or a primal-dual method), we can define its localized version by applying the algorithm to problem \eqref{def:problem_localized} with $b'=b_{V'} + p_{V'} - A_{V',\vec E'^C} x^\star(b)_{\vec E'^C}$. We now provide a general analysis of the comparison between the performance of a given algorithm and its localized version.

\subsection{Bias-variance decomposition and trade-off}
\label{sec:Error analysis: bias-variance decomposition}
The error committed by a chosen local algorithm $T'_{b+p}$ on the subgraph $\vec G'$ after $t \ge 1$ iterations is given by
$$
	\operatorname{Error}(\vec G',t)
	:= x^\star(b+p) - T'^t_{b+p}(x^\star(b)).
$$
The analysis that we give is based on the error decomposition
$$
	\operatorname{Error}(\vec G',t)
	= \operatorname{Bias}(\vec G')
	+ \operatorname{Var}(\vec G',t),
$$
with
\begin{align}
	\operatorname{Bias}(\vec G')
	&:= x^\star(b+p) - \left\{\lim_{t\rightarrow\infty}
	T'^t_{b+p}(x^\star(b))\right\},\label{def:bias}\\
	\operatorname{Var}(\vec G',t)
	&:=
	\left\{\lim_{t\rightarrow\infty} T'^t_{b+p}(x^\star(b)) \right\}
	- T'^t_{b+p}(x^\star(b)).\label{def:variance}
\end{align}
This decomposition resembles the bias-variance decomposition in statistics, which motivates the choices of the terminology we use. In statistics, the term ``bias" typically refers to the \emph{approximation} error that is made from the simplifying assumptions built into the learning method, while the term ``variance" refers to the \emph{estimation} error that is made from the fluctuations of the learning method (trained on a given sample size) around its mean. Analogously, in our setting the bias term \eqref{def:bias} represents the error that is made from the model restriction that we consider, namely, the localization of the chosen algorithmic procedure. The variance term \eqref{def:variance} represents the error that is made by the deviation of the estimate given by the localized algorithm (at a given time $t$) from the optimal solution of the localized model.
The bias term is algorithm-independent and it characterizes the error that we commit by localizing the optimization procedure per se, as a function of the subgraph $\vec G'$. The variance term depends on the algorithm that we run on $\vec G'$ and on time.

We now show that in some regimes the bias introduced by localization can be exploited to lower the computational complexity associated to the variance term and yield savings for local algorithms. This bias-variance trade-off further motivates our analogy with statistics and the terminology we use.

Assume that we want to compare the computational complexity of a global algorithm versus its localized counterpart to achieve a prescribed error accuracy $\varepsilon >0$.
Let $t':=\operatorname{min}\{t>0:\|\operatorname{Var}(\vec G',t)\|\le \varepsilon\}$ be the minimal number of iterations that allows the localized algorithm $T'_{b+p}(x^\star(b))$ to achieve a variance error (measured in the $\ell_2$-norm) less than $\varepsilon$. Let $\kappa(\vec G', 1/\varepsilon)$ be the computational cost needed to run the localized algorithm for $t'$ iterations. 
Typically, $\kappa(\vec G', 1/\varepsilon)$ scales polynomially with $|V'|$ and $|\vec E'|$; it scales polynomially, logarithmically, or double-logarithmically with $1/\varepsilon$, depending on the regularity assumptions for the optimization problem and on the algorithmic procedure being used. For the sake of illustration, we assume that $\kappa(\vec G', 1/\varepsilon)$ is only a function of $|V'|$ and $1/\varepsilon$, so we write $\kappa(|V'|, 1/\varepsilon)$. The global algorithm ($\vec G'=\vec G$) has zero bias, so the computational cost $\kappa(|V|, 1/\varepsilon)$ guarantees to achieve $\|\operatorname{Error}(\vec G,t)\|\le\varepsilon$. The localized algorithm, on the other hand, has a non-zero bias due to localization. However, it will have a smaller computational cost for the variance term, as the algorithm runs on a subgraph of $\vec G$. To investigate the regime in which the added bias yields computational savings for the overall error term, we proceed as follows. By the triangle inequality, we have
$$
	\|\operatorname{Error}(\vec G',t)\|
	\le \|\operatorname{Bias}(\vec G')\|
	+ \|\operatorname{Var}(\vec G',t)\|.
$$
Assume that we can prove a bound of the following form:
\begin{align}
	\| \operatorname{Bias}(\vec G') \| \le \varphi \frac{1}{|V'|^\theta},
	\label{def:bias_ansatz}
\end{align}
for given universal constants $\varphi,\theta >0$. Then, the choice $|V'|\ge (2\varphi/\varepsilon)^{1/\theta}$ guarantees that $\| \operatorname{Bias}(\vec G') \| \le \varepsilon/2$.
By requiring $\|\operatorname{Var}(\vec G',t)\|\le \varepsilon/2$ we find that the computational cost $\kappa((2\varphi/\varepsilon)^{1/\theta}, 2/\varepsilon)$ will guarantee that the localized algorithm also achieves $\|\operatorname{Error}(\vec G',t)\|\le\varepsilon$. This argument shows that the localized algorithm is preferable when $\kappa((2\varphi/\varepsilon)^{1/\theta}, 2/\varepsilon) < \kappa(|V|, 1/\varepsilon)$, for instance. The regime where localized algorithms yield computational savings depends on a variety of factors: the regularity assumptions, the algorithm being chosen, the dimension of the original graph, and the required error tolerance. In the following we provide concrete settings where we can establish  \eqref{def:bias_ansatz} both theoretically and numerically, and investigate the computational savings due to the bias-variance trade-off. The ideas here illustrated suggest a general framework to study the trade-off between accuracy and complexity for local algorithms in network optimization.

\subsection{Well-conditioned setting and projected gradient descent}
\label{sec:Well-conditioned setting and projected gradient descent}
We now consider one of the simplest settings where we can establish locality and illustrate the computational savings that can be achieved by localizing a given optimization procedure. Henceforth, let each function $f_e$ be $\alpha$-strongly convex and $\beta$-smooth for some given parameters $\alpha,\beta \in (0,\infty)$, i.e.,
$
	0<\alpha \le \frac{d^2 f_e(x)}{d x^2} \le \beta < \infty
$
for any $x\in\mathbb{R}, e\in \vec E$.
In this well-conditioned case it is well-known that gradient descent converges with a number of iterations that scales logarithmically with $1/\varepsilon$. Even in this favorable case, however, the computational savings achieved by a local algorithm can be considerable, as we now show.

For the sake of illustration, let us consider projected gradient descent.
Let $\Pi_\mathcal{X}$ be the projection operator on a set $\mathcal{X}$, defined as
$
	\Pi_\mathcal{X}(x) := \operatorname{argmin}_{u\in\mathcal{X}} \| x - u \|.
$
The localized projected gradient descent on $\vec G'$ is defined as follows.

\begin{definition}[Localized projected gradient descent]
Given $x\in\mathcal{X}_b'$, define the set $\mathcal{X}'_b(x):=\{u\in \mathbb{R}^{\vec E'}: A_{V',\vec E'} u_{\vec E'}=b_{V'} - A_{V',\vec E'^C} x_{\vec E'^C} \}$. Localized projected gradient descent with step size $\eta>0$ is the local algorithm defined by
$$
	T'_{b}(x)_{e} =
	\begin{cases}
		\Pi_{\mathcal{X}'_b(x)}(x_{\vec E'} - \eta \nabla f'(x_{\vec E'}) )_e &\text{if } e\in \vec E',\\
		x_{e} &\text{if } e\in \vec E'^C.
	\end{cases}
$$
\end{definition}

When $\vec G'=\vec G$, this algorithm recovers the global algorithm applied to the whole graph $\vec G$. In the setting we consider, a classical result yields that projected gradient descent with step size $\eta=1/\beta$ converges to the optimal solution for any starting point. This corresponds to condition $(i)$ in Definition \ref{def:localalgorithm}.
The algorithm converges exponentially fast, with
$
	\| T'^t_{b}(x)_{\vec E'} - x'^\star(b')_{\vec E'} \| \le e^{- t/(2Q)} \| x_{\vec E'} - x'^\star(b')_{\vec E'} \|,
$
where $Q=\beta/\alpha$ is the so-called \emph{condition number} (see \cite{B14}[Theorem 3.6], for instance).
Under the (common) assumption that $\| x_{\vec E'} - x'^\star(b')_{\vec E'} \| \le R$, for a certain universal constant $R>0$, the above convergence rate tells us that in order to reach a prescribed level of error accuracy $\varepsilon$, it is sufficient to run projected gradient descent for a number of iterations that does not depend on the dimension (nor on the topology) of the subgraph $\vec G'$ it is applied to, and that only scales logarithmically in $1/\varepsilon$.

The bias-variance tradeoff exploits the fact that the computational cost per iteration \emph{does} depend on the dimension, even if the number of iterations is dimension-free. In the case of projected gradient descent the cost per iteration is dominated by the cost of computing the projection step, and in general this cost scales polynomially with the graph size. To be precise, it can be seen that
\begin{align}
\begin{aligned}
	T'_{b+p}(x)_{\vec E'} 
	=&\ 
	(I-A'^TL'^{+}A')(x_{\vec E'}-\eta \nabla f'(x_{\vec E'}))\\
	&\ + A'^TL'^{+}(b_{V'} + p_{V'} - A_{V',\vec E'^C} x_{\vec E'^C} ),
\end{aligned}
	\label{def:localizedalg}
\end{align}
where $L':=A'A'^T$ is the graph Laplacian of the subgraph $G'$. Computing the pseudoinverse of $L'$ exactly has a cost that scales like $O(|V'|^\omega)$, where $\omega > 2$ is the so-called matrix multiplication constant.\footnote{If we relax the requirement of performing an exact projection, we can consider efficient quasi-linear solvers to approximately compute the inverse of $L'$ up to precision $\delta$ \cite{KoutisMP11}. These solvers have a complexity that scales like $\widetilde O(|\vec{E'}|\log |V'| \log(1/\delta))$. Even in this case, however, the savings of localized algorithms are still considerable in appropriate regimes (essentially, the same argument that we provide in the main text holds with $\omega \approx 1$).}
In this case we have that $\kappa(\vec G', 1/\varepsilon)$ scales like $O(|V'|^\omega + |\vec E'|\log (1/\varepsilon))$. If we consider graphs that have largest degree bounded above by a universal constant $k$, we have $|\vec E'| \le (k/2)|V'|$ and we can use the argument given at the end of Section \ref{sec:Error analysis: bias-variance decomposition} to state the following result.

\begin{proposition}[Computational cost, global versus localized]
\label{prop:Computational cost}
Under the assumptions given in this section, assuming that \eqref{def:bias_ansatz} holds, projected gradient descent is guaranteed to compute a solution with accuracy $\varepsilon$ with a cost that scales as:
\begin{itemize}
\item Global algorithm: $O(|V|^\omega + |V|\log (1/\varepsilon))$;
\item Localized algorithm: $O((1/\varepsilon)^{\omega/\theta} + (1/\varepsilon)^{1/\theta}\log (1/\varepsilon))$.
\end{itemize}
\end{proposition}

This result shows that the computational savings achieved by localized algorithms can be substantial even in a well-conditioned setting, in the case when $|V| \gg (1/\varepsilon)^{1/\theta}$, i.e., when the graph $\vec G$ is very large compared to the inverse of the error tolerance $1/\varepsilon$. In the next two sections we investigate this regime both in theory (for expanders) and in simulations (for expanders and grids).

\subsection{Expander graphs}
\label{sec:Expander graphs}

In this section we assume that $G$ is an expander graph. Using the decay of correlation property in Theorem \ref{thm:Decay of correlation}, we provide upper bounds for the bias and variance terms defined in \eqref{def:bias} and \eqref{def:variance} as a function of the subgraph $\vec G'$ and time $t$.

Let define the \emph{inner boundary} of $\vec G'$ as
$
	\Delta(\vec G') :=
	\left\{ v\in V' : \mathcal{N}(v) \cap  V'^C \neq \varnothing \right\},
$
which represents the subset of vertices in $V'$ that have at least one vertex neighbor outside $V'$ (in the undirected graph $G$).
Let $B\in\mathbb{R}^{V\times V}$ be the \emph{vertex-vertex adjacency matrix} of the undirected graph $G=(V,E)$, which is the symmetric matrix defined as
 $B_{uv}:=1$ if $\{u,v\}\in E$, $B_{uv}:=0$ otherwise.
Being real and symmetric, the matrix $B$ has $n:=|V|$ real eigenvalues which we denote by $\mu_{n} \le \mu_{n-1} \le \cdots \le \mu_2 \le \mu_1$. Let $\mu:=\max\{|\mu_2|,|\mu_{n}|\}$ be the second largest eigenvalue in magnitude of $B$.
The next theorem yields bounds for the bias and variance error terms in the $\ell_2$-norm. The proof of this theorem is in Appendix \ref{app:error localized}.

\begin{theorem}[Error localized algorithm]
\label{thm:error localized}
Let $k_-$ and $k_+$ be, respectively, the minimum and maximum degree of $G$, $Q=\beta/\alpha$, and $\rho:=\frac{Qk_+}{k_-} -1 + \frac{Q}{k_-} \mu$. If $\rho< 1$, then
\begin{align*}
	\| \operatorname{Bias}(\vec G') \|
	&\le 
	\| p \| \, \gamma\,
	\frac{\rho^{d(\Delta(\vec G'),Z)}}{(1-\rho)^2} \, \mathbf{1}_{\vec G'\neq \vec G}
	\quad \ \,(\text{algorithm-free})\\
	\| \operatorname{Var}(\vec G',t) \|
	&\le 
	\| p \|\, c\,
	\frac{e^{- t/(2Q)}}{1-\rho}
	\ \, (\text{projected gradient descent})
\end{align*}
where $\gamma := c
	( 1
	+
	c\sqrt{k_+\!-\!1}
	)$ and $c:=\frac{\sqrt{2k_+}}{k_-} Q$.
\end{theorem}

The bound for the bias decays exponentially with respect to the graph-theoretical distance (i.e., the distance in the undirected graph $G$) between the inner boundary of $\vec G'$, i.e., $\Delta(\vec G')$, and the region where the perturbation $p$ is supported, i.e., $Z\subseteq V$. The rate is governed by the eigenvalue $\mu$, the condition number $Q$, and the maximum/minimum degree of the graph.
The bound for the variance decays exponentially with respect to the running time, with rate proportional to $1/Q$.
We highlight that the constants appearing in the bounds in Theorem \ref{thm:error localized} do not depend on the choice of the subgraph $\vec G'$ of $\vec G$, but depend only on $\mu$, $Q$, $k_+$, and $k_-$ (as the proofs in Appendix \ref{app:error localized} attest, a more refined analysis can yield better constants that do depend on the choice of $\vec G'$). In particular, the same constants apply for the analysis of the global algorithm. In this case, the bias term is zero, so that the error equals the variance component (see the indicator function in Theorem \ref{thm:error localized}).

To investigate the computational savings achieved by localized gradient descent in the case of expanders, we will use Theorem \ref{thm:error localized} to derive a bound for the bias term as in \eqref{def:bias_ansatz} and then invoke Proposition \ref{prop:Computational cost}. To this end, for the sake of simplicity, let $G$ be a $k$-regular graph, where each vertex has $k$ neighbors ($k_+=k_-=k$). Let us introduce a collection of subgraphs that are centered on a given vertex and are parametrized by their radii.
Namely, fix a vertex $v\in V$, let $V'_r:=\{w\in V: d(v,w)\le r\}$ denote the ball of radius $r>0$ around vertex $v\in V$, and let $\vec G_r':=(V_r',\vec E_r')$ be the subgraph of $\vec G$ that has vertex set $V_r'$ and induced edge set $\vec E_r'$.
Consider a perturbation vector $p\in\mathbb{R}^V$ that is supported on $Z:=V'_z$, for a fixed $z\ge 1$.
If we run the localized algorithm on $\vec G_r'$, with $r> z$, using the trivial bound $|V'_r| \le k^r$, Theorem \ref{thm:error localized} yields bound \eqref{def:bias_ansatz} with $\varphi=\| p \| \gamma \frac{1}{(1-\rho)^2\rho^z}$ and $\theta = \log(1/\rho)/\log k$. As $|E'_r|\le (k / 2)|V'_r|$, we have that the global projected gradient descent achieves error tolerance $\varepsilon$ with a cost that scales like $O(|V|^\omega + (k/2)|V|\log (1/\varepsilon))$, while its localized counterpart achieves the same accuracy with a cost that scales like $O((2\varphi/\varepsilon)^{\omega/\theta} + (k/2)(2\varphi/\varepsilon)^{1/\theta}\log (2/\varepsilon))$.

\begin{remark}[General graph topologies]
While the theoretical analysis that we present only holds for expanders, the main idea applies to any graph topology. In general, one needs to use Theorem \ref{thm:comparisontheoremnetworkflow} to establish locality results in the same spirit of Theorem \ref{thm:Decay of correlation} and Theorem \ref{thm:error localized}, possibly establishing different rates of decay (i.e., not exponential) for different graphs, as discussed in Section \ref{sec:Correlation in terms of graph Laplacians} and as we will see in Section \ref{sec:Numerical Simulations}.
\end{remark}

\subsection{Numerical Simulations}
\label{sec:Numerical Simulations}

Let us consider the quadratic problem obtained with the choice $f_e(x)=\tau_ex^2/2$ for any $e\in \vec{E}$, where each $\tau_e$ is an independent sample from the uniform distribution in $[1,2]$. 
The function $f=\sum_{e\in\vec E} f_e$ is $\alpha$-strongly convex and $\beta$-smooth with $\alpha=1$ and $\beta =2$, as the Hessian $\nabla^2 f$ is diagonal with entries $\frac{d^2 f_e(x)}{d x^2}=\tau_e$. In this case the optimal solution of the optimization problem can be computed analytically. If we define $\Sigma:=(\nabla^2 f)^{-1}$ and let $\Sigma':=\Sigma_{V',V'}$ be the submatrix indexed by $V'$, we have $x^\star(b) = \Sigma A^T (A\Sigma A^T)^{+} b$ and $x'^\star(b') = \Sigma' A'^T (A'\Sigma' A'^T)^{+} b'$. Using these expressions we can compute the bias term \eqref{def:bias}. For an arbitrary $v\in V$, we take $Z=\{w\in V : d(v,w)\le 1\}$. We draw the components of the vector $b$ and the non-zero components of the vector $p$ independently from the uniform distribution in $[-1,1]$, imposing the conditions $\mathbb{1}^Tb =0$ and $\mathbb{1}^Tp =0$ (this is done by modifying an arbitrary component of the randomly-generated vectors to impose that the sum of the components equals zero). Given $r>0$, let $\vec G'_r=(V'_r,\vec E'_r)$ be the subgraph of $\vec G$ that includes all the vertices within a distance $r$ from $v$.

We consider three graphs with $900$ nodes each: a cycle, a two-dimensional square grid with periodic boundary conditions, and a $3$-regular expander (uniformly sampled from the family of $3$-regular graphs).

Figure \ref{fig-bias} shows the behavior of the $\ell_2$-norm of the bias term as a function of the radius $r$ and the size $|V'_r|$ (i.e., number of vertices) of the subgraph $\vec G'_r$ for the three graph topologies of interest.
Each topology gives rise to a different behavior. 
For the cycle graph, the bias error stays constant while $V'_r\neq V$ and drops to zero only when $V'_r=V$. This illustrates the fact that in the cycle there is no decay of correlation as perturbations do not dissipate. Borrowing again terminology from probability theory, we can say that the cycle represents a case of ``long-range dependence." 
For the two-dimensional grid, the bias decays polynomially with $r$ while the subgraph $\vec G'_r$ does not reach the boundaries of $\vec G$ (as we conjectured, see Section \ref{sec:Correlation in terms of graph Laplacians}), and then decays at a faster rate. For the expander, the bias decays exponentially with $r$, as we proved in Theorem \ref{thm:error localized}. With respect to $|V'_r|$, the decay of the bias for both grids and expanders aligns with the polynomial ansatz in \eqref{def:bias_ansatz}. This polynomial decay for expanders was proved in Section \ref{sec:Expander graphs}.

\begin{figure}[h!]
\centering
\includegraphics[width=5.37cm]{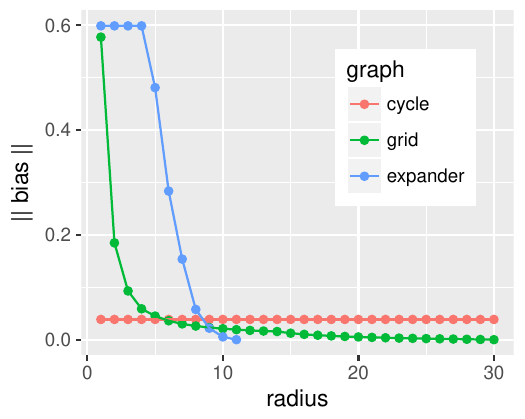}
\hspace{.5cm}
\includegraphics[width=5.37cm]{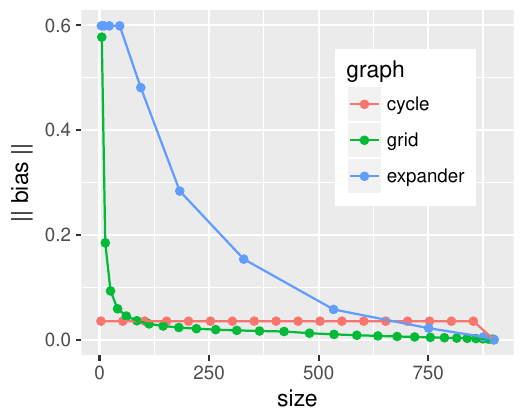}
\caption{Typical realizations of $\| \operatorname{Bias}(\vec G'_r) \|$ as a function of the radius $r$
(left) and the size $|V'_r|$ (right) 
of the subgraph $\vec G'_r=(V'_r,\vec E'_r)$. The value of the bias for the expander is zero for any $r\ge11$, as $V'_r=V$ in this case.
}
\label{fig-bias}
\end{figure}
\FloatBarrier

Figure \ref{fig-variance} shows the behavior of the variance term for the localized projected gradient descent algorithm. As prescribed by classical results, and as we proved more precisely in the case of expanders (Theorem \ref{thm:error localized}), the algorithm converges exponentially fast and the rate of convergence can be upper bounded by a quantity that does not depend on the graph topology nor on the graph size. As described in Section \ref{sec:Well-conditioned setting and projected gradient descent}, the cost incurred by the algorithm to achieve a given accuracy for the variance is seen to increase polynomially with $|V'_r|$.

\begin{figure}[h!]
\centering
\includegraphics[width=5.37cm]{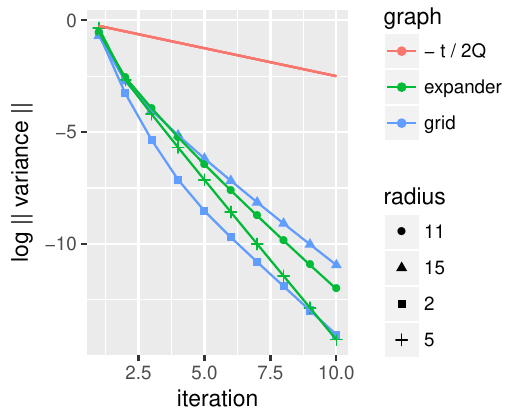}
\hspace{.5cm}
\includegraphics[width=5.37cm]{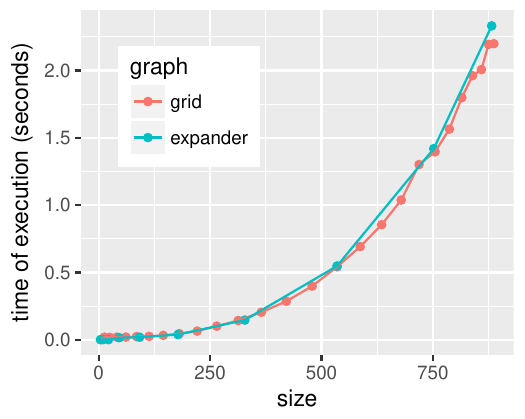}
\caption{(Left) Typical realizations of $\log\| \operatorname{Var}(\vec G'_r,t) \|$ for projected gradient descent as a function of the iteration step $t$, for different graph topologies and different choices of the radius $r$. We do not plot results for the cycle as in that case the algorithm converges in one iteration as there is a unique solution to  $Ax=b$. We also plot the theoretical upper bound $-t/2Q$, with $Q=\beta/\alpha=2$. (Right) Time of execution to run localized projected gradient descent to achieve $\| \operatorname{Var}(\vec G'_r,t) \| \le 10^{-7}$ as a function of $|V'_r|$. We use the \texttt{ginv} function from the \texttt{MASS} library in \texttt{R} to compute the pseudoinverse of the Laplacian matrix $L'$ and run the algorithm as described in \eqref{def:localizedalg}.}
\label{fig-variance}
\end{figure}

\FloatBarrier

Finally, Figure \ref{fig-biasvariance-time} compares the computational cost of global and localized algorithms to achieve the same level of error accuracy. This plot aligns with the findings of Proposition \ref{prop:Computational cost} in the case of expanders. In particular, the plot shows that the computational savings achieved by localized algorithms are considerable in the regime $1/\varepsilon \ll |V|$.

\begin{figure}[h!]
\centering
\includegraphics[width=5.37cm]{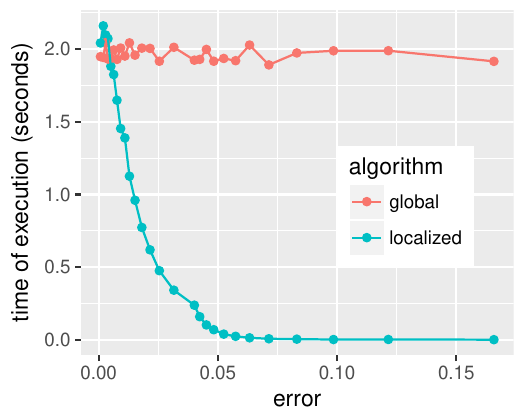}
\hspace{.5cm}
\includegraphics[width=5.37cm]{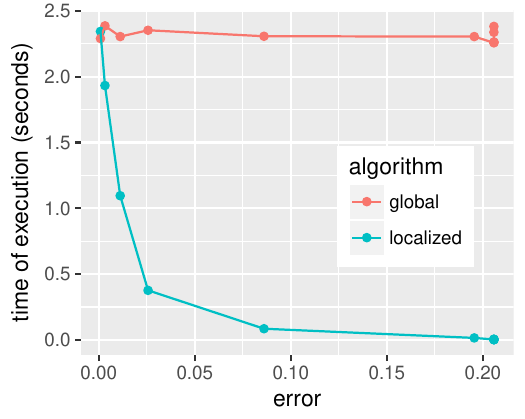}
\caption{Time of execution to run projected gradient descent (gobal and localized) to achieve $\| \operatorname{Error}(\vec G'_r,t) \| \le \varepsilon + 10^{-7}$ as a function of the error tolerance parameter $\varepsilon$, for the grid (left) and the expander (right). As in the well-conditioned setting that we examine the algorithm drives the variance error to zero exponentially fast (in terms of number of iterations, see Figure \ref{fig-variance}), the error is dominated by the bias term. As a consequence, for the localized algorithm we adopt the strategy to choose the smallest radius $r$ such that $\| \operatorname{Bias}(\vec G'_r) \| \le \varepsilon$ and then run the localized algorithms on $\vec G'_r$ for the smallest number of iterations $t$ such that $\| \operatorname{Var}(\vec G'_r,t) \| \le 10^{-7}$.}
\label{fig-biasvariance-time}
\end{figure}


\section{Conclusions}\label{sec:Conclusions}
The main contribution of this paper is to derive a general analogy between natural concepts in probability and statistics (i.e., notions of correlation among random variables, decay of correlation, and bias-variance decomposition and trade-off) and similar notions that can be introduced in optimization. In this paper we have proposed notions of correlation that are based on the sensitivity of optimal points. We have illustrated how decay of correlation (locality) can be established in a canonical network optimization problem (min-cost network flow), and how it can be used to design local algorithms to compute the solution of a problem after localized perturbations are made to the system. In principle, the framework that we propose can be applied to \emph{any} optimization problem. The key point is to establish decay of correlation for the problem at hand, deriving results that are analogous to the ones established in Section \ref{sec:Decay of correlation for expanders}.
In the case of the min-cost network flow problem, we showed that establishing locality reduces to bounding the discrete derivative of the Green's function of the diffusion random walk, as described in Theorem \ref{thm:comparisontheoremnetworkflow}. We proved an exponential decay for the correlation in expander graphs (as a function of the distance from the perturbation), and we conjectured and provided numerical evidence for a polynomial decay in grid-like topologies. 
Once results on locality are established for the particular problem at hand, these results translate into a bound for the bias term of the error decomposition of localized algorithms, as the one that we give in Section \ref{sec:Expander graphs}. The analysis of the variance term, on the other hand, depends on the algorithm that one wants to localize. This part is not technically difficult, as it amounts to analyzing the performance of the chosen algorithm when applied to a subgraph with frozen boundary conditions.
Establishing locality in more general settings and problems, and extending the ideas here presented to \emph{cold-start} scenarios (where one wants to compute the optimal solution of an optimization problem starting from possibly any initial condition) remain open questions for future investigation.

\bibliographystyle{amsplain}
\bibliography{bib}

\providecommand{\bysame}{\leavevmode\hbox to3em{\hrulefill}\thinspace}
\providecommand{\MR}{\relax\ifhmode\unskip\space\fi MR }
\providecommand{\MRhref}[2]{%
  \href{http://www.ams.org/mathscinet-getitem?mr=#1}{#2}
}
\providecommand{\href}[2]{#2}
\begin{thebibliography}{10}

\bibitem{Ahuja:MinCostFlow}
Ravindra~K. Ahuja, Thomas~L. Magnanti, and James~B. Orlin, \emph{{Network
  Flows: Theory, Algorithms, and Applications}}, Prentice Hall, 1993.

\bibitem{opac-b1082769}
Arthur~E. Albert, \emph{Regression and the moore-penrose pseudoinverse},
  Mathematics in science and engineering, Academic Press, 1972.

\bibitem{aldous-fill-2014}
David Aldous and James~Allen Fill, \emph{Reversible markov chains and random
  walks on graphs}, 2002, Unfinished monograph, recompiled 2014, available at
  \url{http://www.stat.berkeley.edu/\~aldous/RWG/book.html}.

\bibitem{reopt}
Claudia Archetti, Luca Bertazzi, and M.~Grazia Speranza, \emph{Reoptimizing the
  traveling salesman problem}, Networks \textbf{42} (2003), no.~3, 154--159.

\bibitem{Belkin2004}
Mikhail Belkin and Partha Niyogi, \emph{Semi-supervised learning on riemannian
  manifolds}, Machine Learning \textbf{56} (2004), no.~1, 209--239.

\bibitem{BT97}
Dimitri~P. Bertsekas and John~N. Tsitsiklis, \emph{Parallel and distributed
  computation: Numerical methods}, Athena Scientific, 1997.

\bibitem{boot63}
J.~C.~G. Boot, \emph{On sensitivity analysis in convex quadratic programming
  problems}, Operations Research \textbf{11} (1963), no.~5, 771--786.

\bibitem{BMS08}
Guy Bresler, Elchanan Mossel, and Allan Sly, \emph{Reconstruction of markov
  random fields from samples: Some observations and algorithms}, Approximation,
  Randomization and Combinatorial Optimization. Algorithms and Techniques,
  Lecture Notes in Computer Science, vol. 5171, Springer Berlin Heidelberg,
  2008, pp.~343--356.

\bibitem{B14}
S{\'e}bastien Bubeck, \emph{Convex optimization: Algorithms and complexity},
  Foundations and Trends{\textregistered} in Machine Learning \textbf{8}
  (2015), no.~3-4, 231--357.

\bibitem{Chung:1997}
F.~R.~K. Chung, \emph{Spectral graph theory}, American Mathematical Society,
  1997.

\bibitem{Dob70}
R.~L. Dobru{\v{s}}in, \emph{Definition of a system of random variables by means
  of conditional distributions}, Teor. Verojatnost. i Primenen. \textbf{15}
  (1970), 469--497.

\bibitem{SAW12}
John~C. Duchi, Alekh Agarwal, and Martin~J. Wainwright, \emph{Dual averaging
  for distributed optimization: Convergence analysis and network scaling.},
  IEEE Trans. Automat. Contr. \textbf{57} (2012), no.~3, 592--606.

\bibitem{fiacco83}
Anthony~V. Fiacco, \emph{Introduction to sensitivity and stability analysis in
  nonlinear programming / anthony v. fiacco}, Academic Press New York, 1983.

\bibitem{FPRS07}
F.~Fouss, A.~Pirotte, J.-M. Renders, and M.~Saerens, \emph{Random-walk
  computation of similarities between nodes of a graph with application to
  collaborative recommendation}, Knowledge and Data Engineering, IEEE
  Transactions on \textbf{19} (2007), no.~3, 355--369.

\bibitem{Friedrich:2010yu}
Tobias Friedrich and Thomas Sauerwald, \emph{The cover time of deterministic
  random walks}, Computing and Combinatorics, vol. 6196, Springer Berlin
  Heidelberg, 2010, pp.~130--139.

\bibitem{GGW14}
David Gamarnik, David~A. Goldberg, and Theophane Weber, \emph{Correlation decay
  in random decision networks}, Mathematics of Operations Research \textbf{39}
  (2014), no.~2, 229--261.

\bibitem{GSW12}
David Gamarnik, Devavrat Shah, and Yehua Wei, \emph{Belief propagation for
  min-cost network flow: Convergence and correctness}, Operations Research
  \textbf{60} (2012), no.~2, 410--428.

\bibitem{Geo11}
Hans-Otto Georgii, \emph{Gibbs measures and phase transitions}, second ed., de
  Gruyter Studies in Mathematics, vol.~9, Walter de Gruyter \& Co., Berlin,
  2011.

\bibitem{guddat76}
J{\"u}rgen Guddat, \emph{Stability in convex quadratic parametric programming},
  Mathematische Operationsforschung und Statistik \textbf{7} (1976), no.~2,
  223--245.

\bibitem{Horn:1985:MA:5509}
Roger~A. Horn and Charles~R. Johnson (eds.), \emph{Matrix analysis}, Cambridge
  University Press, New York, NY, USA, 1986.

\bibitem{Johansson:2009}
Bj\"{o}rn Johansson, Maben Rabi, and Mikael Johansson, \emph{A randomized
  incremental subgradient method for distributed optimization in networked
  systems}, SIAM J. on Optimization \textbf{20} (2009), no.~3, 1157--1170.

\bibitem{KoutisMP11}
Ioannis Koutis, Gary~L. Miller, and Richard Peng, \emph{A nearly-m log n time
  solver for sdd linear systems}, FOCS, IEEE, 2011, pp.~590--598.

\bibitem{krantz2002implicit}
S.G. Krantz and H.R. Parks, \emph{The implicit function theorem: History,
  theory, and applications}, The Implicit Function Theorem: History, Theory,
  and Applications, Birkh{\"a}user, 2002.

\bibitem{rasmussushant15}
Rasmus Kyng, Anup Rao, Sushant Sachdeva, and Daniel~A. Spielman,
  \emph{Algorithms for lipschitz learning on graphs.}, COLT, JMLR Workshop and
  Conference Proceedings, vol.~40, 2015, pp.~1190--1223.

\bibitem{Kyparisis1986}
Jerzy Kyparisis, \emph{Uniqueness and differentiability of solutions of
  parametric nonlinear complementarity problems}, Mathematical Programming
  \textbf{36} (1986), no.~1, 105--113.

\bibitem{lawler2010random}
G.F. Lawler and V.~Limic, \emph{Random walk: A modern introduction}, Cambridge
  Studies in Advanced Mathematics, Cambridge University Press, 2010.

\bibitem{lovasz1993random}
L.~Lov{\'a}sz, \emph{Random walks on graphs: A survey}, Combinatorics, Paul
  Erdos is Eighty \textbf{2} (1993), no.~1, 1--46.

\bibitem{MVR10}
Ciamac~C. Moallemi and Benjamin Van~Roy, \emph{Convergence of min-sum
  message-passing for convex optimization}, Information Theory, IEEE
  Transactions on \textbf{56} (2010), no.~4, 2041--2050.

\bibitem{MRS10}
Damon Mosk-Aoyama, Tim Roughgarden, and Devavrat Shah, \emph{Fully distributed
  algorithms for convex optimization problems}, SIAM Journal on Optimization
  \textbf{20} (2010), no.~6, 3260--3279.

\bibitem{Phu2001}
H.X. Phu and N.D. Yen, \emph{On the stability of solutions to quadratic
  programming problems}, Mathematical Programming \textbf{89} (2001), no.~3,
  385--394.

\bibitem{RvH14}
Patrick Rebeschini and Ramon van Handel, \emph{Comparison theorems for {G}ibbs
  measures}, Journal of Statistical Physics \textbf{157} (2014), no.~2,
  234--281.

\bibitem{RvH15}
\bysame, \emph{Can local particle filters beat the curse of dimensionality?},
  Ann. Appl. Probab. \textbf{25} (2015), no.~5, 2809--2866.

\bibitem{Robinson1979}
Stephen~M. Robinson, \emph{Generalized equations and their solutions, {P}art
  {I}: {B}asic theory}, pp.~128--141, Springer Berlin Heidelberg, Berlin,
  Heidelberg, 1979.

\bibitem{TJ02}
Sekhar~C. Tatikonda and Michael~I. Jordan, \emph{Loopy belief propagation and
  {G}ibbs measures}, Proc. UAI, vol.~18, 2002, pp.~493--500.

\bibitem{citeulike:12634920}
Nisheeth~K. Vishnoi, \emph{Lx = b}, Foundations and Trends{\textregistered} in
  Theoretical Computer Science \textbf{8} (2013), no.~1--2, 1--141.

\bibitem{WOJ13a}
E.~Wei, A.~Ozdaglar, and A.~Jadbabaie, \emph{A distributed {N}ewton method for
  network utility maximization-{I}: Algorithm}, IEEE Transactions on Automatic
  Control \textbf{58} (2013), no.~9, 2162--2175.

\bibitem{WOJ13b}
\bysame, \emph{A distributed {N}ewton method for network utility
  maximization-{II}: Convergence}, IEEE Transactions on Automatic Control
  \textbf{58} (2013), no.~9, 2176--2188.

\bibitem{W06}
Dror Weitz, \emph{Counting independent sets up to the tree threshold},
  Proceedings of the Thirty-eighth Annual ACM Symposium on Theory of Computing,
  ACM, 2006, pp.~140--149.

\bibitem{WD72}
F.~Wu and C.~Desoer, \emph{{Global inverse function theorem}}, IEEE
  Transactions on Circuit Theory \textbf{19} (1972), 199--201.

\end{thebibliography}

\appendix


\section{Hadamard's global inverse theorem}
\label{sec:Hadamard global inverse theorem}

We prove Theorem \ref{thm:comparisontheorem}. The proof relies on Hadamard's global inverse function theorem, which characterizes when a $C^k$ function is a $C^k$ diffeomorphism.
Recall that a function from $\mathbb{R}^m$ to $\mathbb{R}^m$ is said to be \emph{$C^k$} if it has continuous derivatives up to order $k$. A function is said to be a \emph{$C^k$ diffeomorphism} if it is $C^k$, bijective, and its inverse is also $C^k$.

The following important result characterizes when a $C^k$ function is a $C^k$ diffeomorphism.

\begin{theorem}[Hadamard's global inverse function theorem]
\label{thm:Hadamard global inverse theorem}
Let $\Psi$ be a $C^k$ function from $\mathbb{R}^m$ to $\mathbb{R}^m$. Then, $f$ is a $C^k$ diffeomorphism if and only if the following two conditions hold:
\begin{enumerate}
\item The determinant of the differential of $\Psi$ is different from zero at any point, namely,
$|\frac{d}{dz}\Psi (z)| \neq 0$ for any $z\in\mathbb{R}^m$.
\item The function $\Psi$ is \emph{norm coercive}, namely, for any sequence of points $z_1,z_2,\ldots \in \mathbb{R}^m$ with $\| z_k \|\rightarrow\infty$ it holds $\| \Psi(z_k) \|\rightarrow\infty$ (for any choice of the norm $\| \cdot \|$, as norms are equivalent in finite dimension).
\end{enumerate}
\end{theorem}

\begin{proof}
See \cite{WD72}[Corollary of Lemma 2], for instance, or also \cite{krantz2002implicit} for a more general form of this inverse function to study diffeomorphisms on manifolds.
\end{proof}

The following corollary is the backbone behind Theorem \ref{thm:comparisontheorem}.

\begin{lemma}[Diffeomorphism for Lagrangian multipliers map]
\label{lem:Hadamard global inverse theorem}
Let $f:\mathbb{R}^n\rightarrow \mathbb{R}$ be a strongly convex function, twice continuously differentiable.
Let $A\in\mathbb{R}^{m\times n}$ be a given matrix. Define the function $\Phi$ from $\mathbb{R}^n\times\mathbb{R}^m$ to $\mathbb{R}^n\times\mathbb{R}^m$ as
$$
	\Phi
	(
	x,
	\nu 
	)
	:=
	\left( 
	\begin{array}{c}
	\nabla f(x) + A^T\nu \\
	A x 
	\end{array} \right),
$$
for any $x\in\mathbb{R}^n$, $\nu\in\mathbb{R}^m$.
Then, the restriction of the function $\Phi$ to $\mathbb{R}^n\times \operatorname{Im}(A)$ is a $C^1$ diffeomorphism.
\end{lemma}

\begin{proof}
Let us interpret $\Phi$ as the representation of a transformation $\mathcal{T}$ in the standard basis of $\mathbb{R}^n\times \mathbb{R}^m$. Recall the orthogonal decomposition $\mathbb{R}^m = \operatorname{Im}(A) \oplus \operatorname{Ker}(A^T)$. Let the vectors $u_1,\ldots,u_r\in\mathbb{R}^m$ form an orthogonal basis for $\operatorname{Im}(A)$, where $r$ is the rank of $A$, and let the vectors $v_1,\ldots,v_{m-r}\in\mathbb{R}^m$ form an orthogonal basis for $\operatorname{Ker}(A^T)$. Define the orthogonal matrix $Z=[u_1,\ldots,u_r,z_1,\ldots,z_{m-r}]$, which represents a change of basis in $\mathbb{R}^m$. As we have
$$
	\Phi
	(
	x,
	\nu 
	)
	=
	\left( 
	\begin{array}{c}
	\nabla f(x) + A^T Z Z^T \nu \\
	Z Z^T A x 
	\end{array} \right),
$$
then the transformation $\mathcal{T}$ is represented in the standard basis for $\mathbb{R}^n$ and in the basis $Z$ for $\mathbb{R}^m$ by the following map
$$
	\widetilde\Phi
	(
	x,
	\tilde\nu 
	)
	:=
	\left( 
	\begin{array}{c}
	\nabla f(x) + \widetilde A^T \tilde\nu \\
	\widetilde A x 
	\end{array} \right),
$$
where $\widetilde A := Z^T A$. In fact,
$$
	\widetilde\Phi(x,Z^T\nu) = 
	\left( \begin{array}{cc}
	I & \mathbb{O} \\
	\mathbb{O}^T & Z^T 
	\end{array} \right)
	\Phi(x,\nu),
$$
where $I\in \mathbb{R}^{n\times n}$ is the identity matrix, and $\mathbb{O}\in \mathbb{R}^{n\times m}$ is the all-zero matrix. As
$
	A^TZ
	= [A^Tu_1,\ldots,A^Tu_r,\mathbb{O}_{n\times (m-r)}],
$
$$
	\widetilde A = (A^TZ)^T 
	= \left[
	\begin{array}{c}
	B\\
	\mathbb{O}_{(m-r)\times n}
	\end{array} \right],
$$
where $B:=[u_1,\ldots,u_r]^T A\in\mathbb{R}^{r\times n}$. Therefore, the restriction of the transformation $\mathcal{T}$ to the invariant subspace $\mathbb{R}^n\times\operatorname{Im}(A)$ is represented in the standard basis for $\mathbb{R}^n$ and in the basis $\{u_1,\ldots,u_r\}$ for $\operatorname{Im}(A)$ by the following map
$$
	\Psi
	(
	x,
	\xi 
	)
	:=
	\left( 
	\begin{array}{c}
	\nabla f(x) + B^T \xi \\
	B x 
	\end{array} \right).
$$
As the function $f$ is twice continuously differentiable, clearly the function $\Psi$ is continuously differentiable, i.e., $C^1$. We check that the two conditions of Theorem \ref{thm:Hadamard global inverse theorem} are met.

The differential of $\Psi$ evaluated at $(x,\xi)\in\mathbb{R}^n\times\mathbb{R}^r$ is
$$
	J(x,\xi)
	:= 
	\left( \begin{array}{cc}
	\nabla^2 f(x) & B^T \\
	B & \mathbb{O} 
	\end{array} \right).
$$
As $f$ is strongly convex, $\nabla^2 f(x)$ is positive definite so invertible. Then, the determinant of the Jacobian can be expressed as $|J(x,\xi)|= |\nabla^2 f(x)||-B\nabla^2 f(x)^{-1}B^T|$. As $B$ has full row rank by definition, $B\nabla^2 f(x)^{-1}B^T$ is positive definite and we clearly have $|J(x,\xi)| \neq 0$.


To prove that the function $\Psi$ is norm coercive, let us choose $\| \cdot \|$ to be the Euclidean norm and consider a sequence $(x_1,\xi_1),(x_2,\xi_2),\ldots \in \mathbb{R}^n\times\mathbb{R}^r$ with $\| (x_k,\xi_k) \|\rightarrow\infty$. As for any $x\in\mathbb{R}^{n}, \xi\in\mathbb{R}^{r}$ we have $\| (x,\xi) \|^2 = \|x\|^2 + \|\xi\|^2$,
clearly for the sequence to go to infinity one of the following two cases must happen:
\begin{enumerate}[(a)]
\item $\|x_k\|\rightarrow\infty$;
\item $\|x_k\|\le c$ for some $c< \infty$, $\|\xi_k\|\rightarrow\infty$.
\end{enumerate}
Before we consider these two cases separately, let us note that, for any $x\in\mathbb{R}^{n}, \xi\in\mathbb{R}^{r}$,
\begin{align}
	\| \Psi(x,\xi) \|^2 = \| \nabla f(x) + B^T\xi \|^2 + \| Bx \|^2.
	\label{bound}
\end{align}
Let $\alpha > 0$ be the strong convexity parameter, and recall the following definition of strong convexity, for any $x,y\in\mathbb{R}^{n}$,
\begin{align}
	(\nabla f(x)-\nabla f(y))^T(x-y) \ge \alpha \| x-y \|^2.
	\label{strong convexity}
\end{align}

\begin{enumerate}[(a)]
\item
Assume $\|x_k\|\rightarrow\infty$.
Let $P_{\parallel}$ be the projection operator on $\operatorname{Im}(B^T)$, i.e., $P_{\parallel}:=B^T(BB^T)^{-1}B$, and let $P_{\perp}=I-P_{\parallel}$ be the projection operator on $\operatorname{Ker}(B)$, the orthogonal complement of $\operatorname{Im}(B^T)$. As for any $x\in\mathbb{R}^n$ we have the decomposition $x=P_{\parallel}x + P_{\perp}x$ with $(P_{\parallel}x)^TP_{\perp}x =0$, clearly $\|x\|^2=\|P_\parallel x\|^2+\|P_\perp x\|^2$. So, the condition $\|x_k\|\rightarrow\infty$ holds only if one of the two cases happens:
\begin{enumerate}[(i)]
\item $\|P_\parallel x_k\|\rightarrow\infty$;
\item $\|P_\parallel x_k\|\le c$ for some $c< \infty$, $\|P_\perp x_k\|\rightarrow\infty$.
\end{enumerate}

Consider (i) first. Let $x\in\mathbb{R}^n$ so that $P_\parallel x \neq \mathbb{O}$. As $BP_\perp = \mathbb{O}$, from \eqref{bound} we have, for any $\xi\in\mathbb{R}^{r}$, $\| \Psi(x,\xi) \|^2 \ge \| Bx \|^2 = \| B P_\parallel x \|^2$, and
$
	\min_{y\in\mathbb{R}^{n} : y\in \operatorname{Im}(B^T), y\neq \mathbb{O}} \frac{y^TB^TBy}{\|y\|^2} \|P_\parallel x\|^2
	= \lambda \|P_\parallel x\|^2
$
is a lower bound, where $\lambda$ is the minimum eigenvalue of $B^TB$ among those corresponding to the eigenvectors spanning the subspace $\operatorname{Im}(B^T)$. Clearly, if $\lambda\neq 0$ (notice $\lambda \ge 0$ by definition) then the above yields that $\| \Psi(x_k,\xi_k) \| \rightarrow \infty$ whenever $\|P_\parallel x_k\|\rightarrow\infty$. To prove this, assume by contradiction that $\lambda=0$. Then, there exists $y\in\mathbb{R}^{n}$ satisfying $y\in \operatorname{Im}(B^T), y\neq \mathbb{O}$, such that $B^TBy=\lambda y= \mathbb{O}$. As $B^T$ has full column rank, the latter is equivalent to $By=\mathbb{O}$ so that $P_\perp y = y\neq \mathbb{O}$, which contradicts the hypothesis that $y\in \operatorname{Im}(B^T)$.

Consider now the case (ii). Decomposing the gradient on $\operatorname{Im}(B^T)$ and its orthogonal subspace, from \eqref{bound} we have, for any $x\in\mathbb{R}^{n}, \xi\in\mathbb{R}^{r}$,
$
	\| \Psi(x,\xi) \|^2 \ge \| P_\perp\nabla f(x) + P_\parallel\nabla f(x) + B^T\xi \|^2
	= \| P_\perp\nabla f(x) \|^2 + \| P_\parallel\nabla f(x) + B^T\xi \|^2,
$
so that $\| \Psi(x,\xi) \| \ge \| P_\perp\nabla f(x) \|$. 
Choosing $y=P_\parallel x$ in \eqref{strong convexity} we have
$$
	(P_\perp\nabla f(x)-P_\perp\nabla f(P_\parallel x))^TP_\perp x=
	(\nabla f(x)-\nabla f(P_\parallel x))^TP_\perp x \ge \alpha \| P_\perp x \|^2,
$$
and applying Cauchy-Schwarz we get, for any $x$ such that $P_\perp x \neq \mathbb{O}$,
$
	\|P_\perp\nabla f(x)\|
	\ge \alpha \| P_\perp x \| - \| P_\perp\nabla f(P_\parallel x) \|.
$
By assumption $f$ is twice continuously differentiable, so $\nabla f$ is continuous and it stays bounded on a bounded domain. Hence, we can conclude that $\| \Psi(x_k,\xi_k) \| \rightarrow \infty$ if $\|P_\perp x_k\|\rightarrow\infty$ with $(P_\parallel x_k)_{k\ge 1}$ bounded.

%
%

\item Assume $\|\xi_k\|\rightarrow\infty$ and $(x_k)_{k\ge 1}$ bounded. For any $\xi\in\mathbb{R}^{r}$, $\xi \neq \mathbb{O}$, we have
$$
	\| B^T\xi \|^2 = \frac{\xi^TBB^T\xi}{\|\xi\|^2} \|\xi\|^2
	\ge \min_{y\in\mathbb{R}^{r} : y\neq \mathbb{O}} \frac{y^TBB^Ty}{\|y\|^2} \|\xi\|^2
	= \lambda_{\text{min}} \|\xi\|^2,
$$
where $\lambda_{\text{min}}$ is the minimum eigenvalue of $BB^T$, which is strictly positive as $BB^T$ is positive definite by the assumption that $B$ has full row rank. From \eqref{bound} we have
$
	\| \Psi(x,\xi) \|
	\ge \| \nabla f(x) + B^T\xi \|
	\ge \| B^T\xi \| - \| \nabla f(x) \|
	\ge \sqrt{\lambda_{\text{min}}} \|\xi\| - \| \nabla f(x) \|,
$
that, by continuity of $\nabla f$, shows that $\| \Psi(x_k,\xi_k) \| \rightarrow \infty$ if $\|\xi_k\|\rightarrow\infty$ and $(x_k)_{k\ge 1}$ is bounded.
\end{enumerate}
\end{proof}

We now present the proof of Theorem \ref{thm:comparisontheorem}.
\begin{proof}[Proof of Theorem \ref{thm:comparisontheorem}]
The Lagrangian of the optimization problem is the function $\mathcal{L}$ from $\mathbb{R}^\mathcal{V}\times\mathbb{R}^\mathcal{F}$ to $\mathbb{R}$ defined as
$
	\mathcal{L}(x,\nu) := f(x) + \sum_{a\in \mathcal{F}} \nu_a (A^T_a x - b_a),
$
where $A^T_a$ is the $a$-th row of the matrix $A$ and $\nu = (\nu_a)_{a\in\mathcal{F}}$ is the vector formed by the Lagrangian multipliers.
Let us define the function $\Phi$ from $\mathbb{R}^\mathcal{V}\times\mathbb{R}^\mathcal{F}$ to $\mathbb{R}^\mathcal{V}\times\mathbb{R}^\mathcal{F}$ as
$$
	\Phi
	( 
	x,
	\nu 
	)
	:=
	\left( 
	\begin{array}{c}
	\nabla_x \mathcal{L}(x,\nu) \\
	A x 
	\end{array} \right)
	=
	\left( 
	\begin{array}{c}
	\nabla f(x) + A^T\nu \\
	A x 
	\end{array} \right).
$$
For any fixed $\varepsilon\in\mathbb{R}$, as the constraints are linear, the Lagrange multiplier theorem says that for the unique minimizer $x^\star(b(\varepsilon))$ there exists $\nu'(b(\varepsilon))\in\mathbb{R}^\mathcal{F}$ so that
\begin{align}
	\Phi
	( 
	x^\star(b(\varepsilon)),
	\nu'(b(\varepsilon))
	)
	=
	\left( 
	\begin{array}{c}
	\mathbb{0} \\
	b(\varepsilon)
	\end{array} \right).
	\label{diff}
\end{align}
As $A^T(\nu+\mu) = A^T\nu$ for each $\mu\in\operatorname{Ker}(A^T)$, the set of Lagrangian multipliers $\nu'(b(\varepsilon))\in\mathbb{R}^\mathcal{F}$ that satisfies \eqref{diff} is a translation of the null space of $A^T$. We denote the unique translation vector by $\nu^\star(b(\varepsilon))\in\operatorname{Im}(A)$.
By Hadamard's global inverse function theorem, as shown in Lemma \ref{lem:Hadamard global inverse theorem}, the restriction of the function $\Phi$ to $\mathbb{R}^\mathcal{V}\times \operatorname{Im}(A)$ is a $C^1$ diffeomorphism, namely, it is continuously differentiable, bijective, and its inverse is also continuously differentiable.
In particular, this means that the functions $x^\star:b\in\operatorname{Im}(A) \rightarrow x^\star(b)\in\mathbb{R}^\mathcal{V}$ and $\nu^\star:b\in\operatorname{Im}(A)\rightarrow \nu^\star(b)\in\operatorname{Im}(A)$ are continuously differentiable along the subspace $\operatorname{Im}(A)$. Differentiating both sides of \eqref{diff} with respect to $\varepsilon$, we get
$$
	\left( \begin{array}{cc}
	H & A^T \\
	A & \mathbb{0} 
	\end{array} \right)
	\left( \begin{array}{c}
	x'\\
	\tilde \nu
	\end{array} \right)
	=
	\left( \begin{array}{c}
	\mathbb{0}\\
	\frac{d b(\varepsilon)}{d \varepsilon}
	\end{array} \right),
$$
where $H:=\nabla^2 f(x^\star (b(\varepsilon)))$, $x' := \frac{d x^\star(b(\varepsilon))}{d \varepsilon}$, $\tilde \nu := \frac{d \nu^\star(b(\varepsilon))}{d \varepsilon}$.
As the function $f$ is strongly convex, the Hessian $\nabla^2 f(x)$ is positive definite for every $x\in\mathbb{R}^\mathcal{V}$, hence it is invertible for every $x\in\mathbb{R}^\mathcal{V}$. Solving the linear system for $x'$ first, from the first equation $Hx'+ A^T\tilde\nu =\mathbb{0}$ we get $x' = - H^{-1} A^T \tilde\nu$. Substituting this expression in the second equation $Ax' = \frac{d b(\varepsilon)}{d \varepsilon}$, we get $L\tilde \nu = - \frac{d b(\varepsilon)}{d \varepsilon}$, where $L:=AH^{-1}A^T$.
The set of solutions to $L\tilde \nu = - \frac{d b(\varepsilon)}{d \varepsilon}$ can be expressed in terms of the pseudoinverse of $L$ as follows 
$
	\{\tilde \nu\in\mathbb{R}^\mathcal{F} : L\tilde \nu = - \frac{d b(\varepsilon)}{d \varepsilon}\}
	=
	-L^+\frac{d b(\varepsilon)}{d \varepsilon} + \operatorname{Ker}(L).
$
We show that $\operatorname{Ker}(L) = \operatorname{Ker}(A^T)$. We show that $L\nu =\mathbb{0}$ implies $A^T\nu =\mathbb{0}$, as the opposite direction trivially holds. In fact, let $A' := A\sqrt{H^{-1}}$, where $\sqrt{H^{-1}}$ is the positive definite matrix that satisfies $\sqrt{H^{-1}} \sqrt{H^{-1}} = H^{-1}$. The condition $L\nu = A' A'^T \nu = \mathbb{0}$ is equivalent to $A'^T \nu\in\operatorname{Ker}(A')$. At the same time, clearly, $A'^T \nu\in\operatorname{Im}(A'^T)$. However, $\operatorname{Ker}(A')$ is orthogonal to $\operatorname{Im}(A'^T)$, so it must be $A'^T\nu =\mathbb{0}$ which implies $A^T\nu =\mathbb{0}$ as $\sqrt{H^{-1}}$ is positive definite.
As $\operatorname{Im}(L^+) = \operatorname{Ker}(L)^\perp = \operatorname{Ker}(A^T)^\perp = \operatorname{Im}(A)$, we have that $\tilde\nu = -L^+\frac{d b(\varepsilon)}{d \varepsilon}$ is the unique solution to $L\tilde \nu = - \frac{d b(\varepsilon)}{d \varepsilon}$ that belongs to $\operatorname{Im}(A)$. 
Substituting this expression into $x' = - H^{-1} A^T \tilde\nu$, we finally get
$
	x' = H^{-1} A^T L^+ \frac{d b(\varepsilon)}{d \varepsilon}.
$
The proof follows as $\Sigma(b)=H^{-1}$.
\end{proof}


\section{Correlation}\label{app:Correlation}

In this section we provide the proofs of Proposition \ref{prop:condsgaussian} and Lemma \ref{lem:Gaussian} in Section \ref{sec:Notions of correlation in optimization}.

\begin{proof}[Proof of Proposition \ref{prop:condsgaussian}]
Let $\mathcal{I}$ be a finite set, and let $Z\in\mathbb{R}^\mathcal{I}$ be a Gaussian random vector with mean $\rho\in\mathbb{R}^{\mathcal{I}}$ and covariance $\Upsilon\in\mathbb{R}^{\mathcal{I}\times\mathcal{I}}$, not necessarily invertible. Let $\mathcal{V}\subseteq\mathcal{I}$ be not empty, and $\mathcal{F}:=\mathcal{I}\setminus \mathcal{V}$ not empty. Recall that if $z_\mathcal{F} \in \operatorname{Im}(\Upsilon_{\mathcal{F},\mathcal{F}})$, where $\Upsilon_{\mathcal{F},\mathcal{F}}$ is the submatrix of $\Upsilon$ indexed by the rows and columns referring to $\mathcal{F}$, then the conditional mean of $Z_\mathcal{V}$ given $Z_\mathcal{F}=z_\mathcal{F}$ is given by (see \cite{opac-b1082769}[Theorem 9.2.1], for instance):
$
	\mathbf{E}[Z_\mathcal{V} | Z_\mathcal{F}=z_\mathcal{F}] = \rho_\mathcal{V} + \Upsilon_{\mathcal{V},\mathcal{F}} (\Upsilon_{\mathcal{F},\mathcal{F}})^{+} (z_\mathcal{F} - \rho_\mathcal{F}).
$
The statement of the proposition follows if we consider the Gaussian vector $Z:=(X,AX)$ with $Z_\mathcal{V}:=X$, $Z_\mathcal{F}:=AX$, upon noticing that $\rho_{\mathcal{V}}=\mathbf{E}[X]=\mu$, $\rho_{\mathcal{F}}=\mathbf{E}[AX]=A\mu$,
$\Upsilon_{\mathcal{V},\mathcal{F}}=\operatorname{Cov}(X,AX)=\mathbf{E}[X(AX)^T] = \mathbf{E}[XX^TA^T]= \Sigma A^T$, and $\Upsilon_{\mathcal{F},\mathcal{F}}=\operatorname{Cov}(AX,AX)= \mathbf{E}[AXX^TA^T]=A\Sigma A^T$. If $\Sigma$ is invertible, i.e., positive definite, and $A$ is full rank, then also $A\Sigma A^T$ is positive definite and invertible, so $(A\Sigma A^T)^{+}=(A\Sigma A^T)^{-1}$.
\end{proof}

\begin{proof}[Proof of Lemma \ref{lem:Gaussian}]
Given $b\in\mathbb{R}^B$, note that $x^\star_I(b)$ corresponds to the components labeled by $I$ of the solution $x^\star(b)$ of the optimization problem \eqref{original opt problem}, with
$
	A
	:=
	(\mathbb{0},I),
$
where $\mathbb{0}$ is the all-zero matrix in $\mathbb{R}^{B\times I}$ and $I$ is the identity matrix in $\mathbb{R}^{B\times B}$, and $x=(x_I,x_B)^T$. As the matrix $A$ is clearly full row rank, and
$
	H =
	\nabla^2 f(x^\star_I(b)b) = \nabla^2 f(x^\star(b)),
$
Corollary \ref{cor:fullrank} yields
$$
	\frac{d x^\star(b)}{db}
	= \Sigma A^T (A\Sigma A^T)^{-1}
	=
	\left( \begin{array}{c}
	\Sigma_{I,B} (\Sigma_{B,B})^{-1}\\
	\mathbb{1}
	\end{array} \right),
$$
so that $\frac{d x^\star_I(b)}{db} = \Sigma_{I,B} (\Sigma_{B,B})^{-1}$. The matrix identity $\Sigma_{I,B} (\Sigma_{B,B})^{-1}=- (H_{I,I})^{-1} H_{I,B}$ is standard textbook material. However, to show the statement involving the matrix $H$, we proceed from first principles, by applying the first order optimality condition to the restricted problem on $I$. Note that $x_I^\star(x_B)$ is defined by the optimality conditions
$
	\frac{d f (x_I^\star(x_B)x_B)}{d x_I} = \mathbb{0}.
$
Differentiating with respect to $x_B$, we get
$$
	\frac{d^2 f (x_I^\star(x_B)x_B)}{dx_I^2} \frac{d x_I^\star(x_B)}{d x_B}
	+ \frac{d^2 f (x_I^\star(x_B)x_B)}{dx_B dx_I} = \mathbb{0},
$$
or, equivalently,
$
	H_{I,I} \frac{d x_I^\star(x_B)}{d x_B} = - H_{I,B}.
$
The proof is concluded by inverting the matrix $H_{I,I}$, which is invertible as $H$ is positive definite by assumption, so that any principal submatrix of it is also positive definite.
\end{proof}


\section{Graph Laplacians and random walks}\label{sec:Laplacians and random walks}

This section is self-contained and provides several connections between graph Laplacians and random walks on weighted graphs. In particular, the connection between pseudoinverses of graph Laplacians and Green's functions of random walks that we present in Lemma \ref{lem:Laplacians and random walks} below 
is the key that allows us to derive spectral bounds to prove the decay of correlation property in Section \ref{sec:Optimal Network Flow}, as showed in Appendix \ref{app:Decay of correlation}.

Throughout, let $G=(V,E,W)$ be a simple (i.e., no self-loops, and no multiple edges), connected, undirected, weighted graph, where to each edge $\{v,w\}\in E$ is associated a non-negative weight $W_{vw}=W_{wv}>0$, and $W_{vw}=0$ if $\{v,w\}\not\in E$. Let $D$ be a diagonal matrix with entries $d_v=D_{vv}=\sum_{w\in V} W_{vw}$ for each $v\in V$.
For each vertex $v\in V$,
let $\mathcal{N}(v):=\{w\in V: \{v,w\}\in E\}$ be the set of node neighbors of $v$.
In this section we establish several connections between the graph Laplacian $L:=D-W$ and the random walk $X:=(X_t)_{t\ge 0}$ with transition matrix $P:=D^{-1}W$.
Henceforth, for each $v\in V$, let $\mathbf{P}_{v}$ be the law of a time homogeneous Markov chain $X_{0},X_{1},X_2,\ldots$ on $V$ with transition matrix $P$ and initial condition $X_0=v$. Analogously, denote by $\mathbf{E}_{v}$ the expectation with respect to this law. The hitting time to the site $v\in V$ is defined as
$
	T_v := \inf\{t\ge0 : X_t = v\}.
$
Let $\pi$ be the unique stationary distribution of the random walk, namely, $\pi^TP = \pi^T$. By substitution it is easy to check that $\pi_v:=\frac{d_v}{\sum_{v\in V} d_v}$ for each $v\in V$. We adopt the notation $e_v\in\mathbb{R}^V$ to denote the vector whose only non-zero component equals $1$ and corresponds to the entry associated to $v\in V$.

\subsection{Restricted Laplacians and killed random walks}
\label{sec:Restricted Laplacians and killed random walks}
The connection between Laplacians and random walks that we present in Appendix \ref{sec:Pseudoinverse of graph Laplacians and Green's function of random walks} below is established by investigating \emph{restricted} Laplacians and \emph{killed} random walks. Throughout this section, let $\bar z\in V$ be fixed. Let $\bar V:=V\setminus Z$, $\bar E:=E\setminus \{\{u,v\}\in E: u\in Z \text{ or } v\in Z\}$ and consider the graph $\bar G:=(\bar V,\bar E)$. 
Define $\bar W$ and $\bar D$ as the matrix obtained by removing the $\bar z$-th row and $\bar z$-th column from $W$ and $D$, respectively.
Then, $\bar L:=\bar D-\bar W \in \mathbb{R}^{\bar V \times \bar V}$ represents the so-called restricted Laplacian that is obtained by removing the $\bar z$-th row and $\bar z$-th column from $L$. Define $\bar P:=\bar D^{-1}\bar W \in \mathbb{R}^{\bar V \times \bar V}$. It is easy to check that $\bar P$ corresponds to the transition matrix of the transient part of the random walk that is obtained from $X$ by adding a cemetery at site $\bar z$. Creating a cemetery at $\bar z$ means modifying the walk $X$ so that $\bar z$ becomes a recurrent state, i.e., once the walk is in state $\bar z$ it will go back to $\bar z$ with probably $1$. This is clearly done by replacing the $\bar z$-th row of $P$ by a row with zeros everywhere but in the $\bar z$-th coordinate, where the entry is equal to $1$. Note that $\bar L$ does \emph{not} correspond to the graph Laplacian of the graph $(\bar V, \bar E, \bar W)$, as $\bar D_{vv} = \sum_{u\in V} W_{vu}$ that does not equal $\sum_{u\in \bar V} \bar W_{vu}$ if $v\in\mathcal{N}(\bar z)$. For the same reason, $\bar P$ does \emph{not} represent the ordinary random walk on $(\bar V, \bar E, \bar W)$.

The relation between the transition matrix $\bar P$ of the killed random walk and the law of the random walk $X$ itself is made explicit in the next proposition. We omit the proof, which follows easily from the Markov property of $X$.

\begin{proposition}\label{prop:killedrv}
For $v,\!w\!\in\!\bar V$, $t\!\ge\! 0$, 
$
	\bar P^t_{vw}
	\!=\!
	\mathbf{P}_v(X_t\!=\!w,T_{\bar z}>t).
$
\end{proposition}


The following proposition relates the inverse of the reduced Laplacian $\bar L$ with the Green function of the killed random walk, namely, the function $(u,v)\in \bar V\times \bar V \rightarrow \sum_{t=0}^\infty \bar P^t_{uv}$, and with the hitting times of the original random walk $X$.

\begin{proposition}\label{prop:reducedlapandgreenfunction}
For each $v,w\in \bar V$, we have
$$
	\bar L^{-1}_{vw} 
	= \frac{1}{d_w} \sum_{t = 0}^\infty \bar P^t_{vw}
	= \bar L^{-1}_{ww}\mathbf{P}_v(T_w<T_{\bar z}),
$$
and
$
	\bar L^{-1}_{ww}
	= \frac{1}{d_w} \mathbf{E}_w [\sum_{t = 0}^{T_{\bar z}} \mathbf{1}_{X_t=w} ].
$
\end{proposition}

\begin{proof}
Let us first assume that $\bar G$ is connected. The matrix $\bar P$ is sub-stochastic as, clearly, if $v \not\in \mathcal{N}(\bar z)$ then $\sum_{w\in V} \bar P_{vw} = 1$, while if $v \in \mathcal{N}(\bar z)$ then $\sum_{w\in V} \bar P_{vw} < 1$. Then $\bar P$ is irreducible (in the sense of Markov chains, i.e., for each $v,w\in \bar V$ there exists $t$ to that $\bar P^t_{vw}\neq 0$) and the spectral radius of $\bar P$ is strictly less than $1$ (see Corollary 6.2.28 in \cite{Horn:1985:MA:5509}, for instance), so that the Neumann series $\sum_{t=0}^\infty \bar P^t$ converges.
The Neumann series expansion for $\bar L^{-1}$ yields
$
	\bar L^{-1}
	= \sum_{t=0}^\infty (I-\bar D^{-1} \bar L)^t \bar D^{-1}
	= \sum_{t=0}^\infty \bar P^t \bar D^{-1},
$
or, entry-wise, $\bar L^{-1}_{vw} = \frac{1}{d_w} \sum_{t = 0}^\infty \bar P^t_{vw}$. As $\bar P^t_{vw}=\mathbf{P}_v(X_t=w,T_{\bar z}>t)$ by Proposition \ref{prop:killedrv}, by the Monotone convergence theorem we can take the summation inside the expectation and get
$$
	\sum_{t = 0}^\infty \bar P^t_{vw} 
	= \sum_{t = 0}^\infty \mathbf{E}_v[\mathbf{1}_{X_t=w}\mathbf{1}_{T_{\bar z}>t}]
	= \mathbf{E}_v\left[\sum_{t = 0}^{T_{\bar z}} \mathbf{1}_{X_t=w}\right],
$$
where in the last step we used that $X_{T_{\bar z}}=\bar z$ and $\bar z\neq w$. Recall that if $S$ is a stopping time for the Markov chain $X:=X_0,X_1,X_2,\ldots$, then by the strong Markov property we know that, conditionally on $\{S<\infty\}$ and $\{X_S=w\}$, the chain $X_S,X_{S+1},X_{S+2},\ldots$ has the same law as a time-homogeneous Markov chain $Y:=Y_{0},Y_1,Y_2,\ldots$ with transition matrix $P$ and initial condition $Y_0=w$, and $Y$ is independent of $X_0,\ldots,X_S$. The hitting times $T_w$ and $T_{\bar z}$ are two stopping times for $X$, and so is their minimum $S:=\min\{T_w, T_{\bar z}\}$. As either $X_S=w$ or $X_S=\bar z$, we have
$$
	\mathbf{E}_v\left[\sum_{t = 0}^{T_{\bar z}} \mathbf{1}_{X_t=w}\right]
	= \mathbf{E}_v\left[\sum_{t = 0}^{T_{\bar z}} \mathbf{1}_{X_t=w} \Bigg\vert X_S=w\right]
	\mathbf{P}_v(X_S=w),
$$
where we used that, conditionally on $\{X_S=\bar z\}=\{T_w > T_{\bar z}\}$, clearly $\sum_{t = 0}^{T_{\bar z}} \mathbf{1}_{X_t=w} = 0$. Conditionally on $\{X_S=w\}=\{T_w < T_{\bar z}\}=\{S=T_w\}$, we have $T_{\bar z}=S + \inf\{t\ge 0:X_{S+t}=\bar z\}$, and the strong Markov property yields (note that the event $\{S<\infty\}$ has probability one from any starting point, as the graph $G$ is connected and the chain will almost surely eventually hit either $w$ or $\bar z$) that
$
	\mathbf{E}_v[\sum_{t = 0}^{T_{\bar z}} \mathbf{1}_{X_t=w} \vert X_S\!=\!w]
$
can be written as
$
	\mathbf{E}_v[\sum_{t = 0}^{\inf\{t\ge 0:X_{S+t}=\bar z\}} 
	\!\mathbf{1}_{X_{S+t}=w} \vert X_S\!=\!w]
	\!=\!
	\mathbf{E}_w[\sum_{t = 0}^{T_{\bar z}} \mathbf{1}_{X_t=w} ].
$
Putting everything together we have
$
	\bar L^{-1}_{vw}
	= \frac{1}{d_w} \mathbf{E}_w[\sum_{t = 0}^{T_{\bar z}} \mathbf{1}_{X_t=w} ]
	\mathbf{P}_v(T_w<T_{\bar z}).
$
As $\mathbf{P}_w(T_w<T_{\bar z})=1$, clearly
$
	\bar L^{-1}_{ww}
	= \frac{1}{d_w} \mathbf{E}_w[\sum_{t = 0}^{T_{\bar z}} \mathbf{1}_{X_t=w} ]
$
so that
$
	\bar L^{-1}_{vw}
	= \bar L^{-1}_{ww}
	\mathbf{P}_v(T_w<T_{\bar z}).
$
The argument just presented extends easily to the case when $\bar G$ is not connected. In fact, in this case the matrix $\bar P$ has a block structure, where each block corresponds to a connected component and to a sub-stochastic submatrix, so that the argument above can be applied to each block separately.
\end{proof}

The following result relates the inverse of the reduced Laplacian $\bar L$ with the pseudoinverse of the Laplacian $L$, which we denote by $L^+$. It is proved in \cite{FPRS07}[Appendix B].


\begin{proposition}\label{prop:connectionlaplacians}
For $v,w\in \bar V$,
$
	\bar L^{-1}_{vw}
	= (e_{v}-e_{\bar z})^T 
	L^{+} (e_{w}-e_{\bar z}).
$
\end{proposition}

Proposition \ref{prop:reducedlapandgreenfunction} and Proposition \ref{prop:connectionlaplacians} allow us to relate the quantity $L^+$ to the \emph{difference} of the Green's function of the random walk $X$, as we discuss next.

\subsection{Laplacians and random walks}
\label{sec:Pseudoinverse of graph Laplacians and Green's function of random walks}
We now relate the Moore-Penrose pseudoinverse of the Laplacian $L:=D-W$ with the Green's function
$
	(u,v)\in V \times V
	\rightarrow
	\sum_{t=0}^\infty P^t_{uv} =
	\mathbf{E}_u[
	\sum_{t = 0}^\infty 
	\mathbf{1}_{X_t=v}]
$
of the random walk, which represents the expected number of times the Markov chain $X$ visits site $v$ when it starts from site $u$.
Notice that as the graph $G$ is finite and connected, then the Markov chain $X$ is recurrent and the Green's function itself equals infinity for any $u,v\in V$. In fact, the following result involves \emph{differences} of the Green's function, not the Green's function itself.

\begin{lemma}
\label{lem:Laplacians and random walks}
For any $u,v,w,z\in V$, we have
\begin{align*}
	(e_{u}-e_{v})^T L^{+} (e_{w}-e_{z})
	&=
	\sum_{t=0}^\infty (e_{u}-e_{v})^T P^{t} 
	\left( \frac{e_{w}}{d_w}-\frac{e_{z}}{d_z} \right).
\end{align*}
\end{lemma}

\begin{proof}
Using first Proposition \ref{prop:connectionlaplacians} and then Proposition \ref{prop:reducedlapandgreenfunction} we obtain, for any $u,v,w,z\in V$ (choose $\bar z$ to be $z$ in Appendix \ref{sec:Restricted Laplacians and killed random walks}),
$
	(e_{u}-e_{v})^T L^{+} (e_{w}-e_{z})
	= 
	(e_{u}-e_{z})^T L^{+} (e_{w}-e_{z})
	-(e_{v}-e_{z})^T L^{+} (e_{w}-e_{z})
	=\bar L^{-1}_{uw} - \bar L^{-1}_{vw}
	= (e_{w}-e_{z})^T L^{+} (e_{w}-e_{z})
	\left\{
	\mathbf{P}_u(T_w<T_z) -
	\mathbf{P}_v(T_w<T_z)
	\right\}.
$
From (3.27) in the proof of Proposition 3.10 in Chapter 3 in \cite{aldous-fill-2014}, upon identifying $v\rightarrow u,x\rightarrow v,v_0\rightarrow w,a\rightarrow z$, we immediately have the following relation between the difference of potentials and hitting times of the random walk $X$:
$
	\mathbf{P}_u(T_w<T_z) -
	\mathbf{P}_v(T_w<T_z)
	=
	\pi_{w}\mathbf{P}_{w}(T_z<T^+_w)
	\left\{\mathbf{E}_uT_z - \mathbf{E}_vT_z + \mathbf{E}_vT_w- \mathbf{E}_uT_w\right\},
$
where $\pi_v:=\frac{d_v}{\sum_{v\in V} d_v}$ is the $v$-th component of the stationary distribution of the random walk $X$, and $
	T^+_v := \inf\{t\ge 1 : X_t = v\}.
$
From Corollary 8 in Chapter 2 in \cite{aldous-fill-2014}, we have
$$
	\pi_{w}\mathbf{P}_{w}(T_z<T^+_w)
	= 
	\begin{cases}
		\frac{1}{\mathbf{E}_wT_z+\mathbf{E}_zT_w} &\text{if } w\neq z,\\
		\pi_w &\text{if } w= z,
	\end{cases}
$$
and we recall the connection between commute times and effective resistance (see, for example, Corollary 3.11 in \cite{aldous-fill-2014}):
$
	\mathbf{E}_wT_z+\mathbf{E}_zT_w
	= (e_{w}-e_{z})^T L^{+} (e_{w}-e_{z}) \sum_{v\in V} d_v.
$
Lemma 3.3 in \cite{Friedrich:2010yu} yields
$
	\mathbf{E}_uT_z - \mathbf{E}_vT_z
	= \frac{1}{\pi_z} \sum_{t=0}^\infty (P^t_{vz} - P^t_{uz}),
$
and
$
	\mathbf{E}_uT_w - \mathbf{E}_vT_w
	= \frac{1}{\pi_w} \sum_{t=0}^\infty (P^t_{vw} - P^t_{uw}).
$
The proof of the lemma follows by combining everything together.
\end{proof}

We have the following corollary of Lemma \ref{lem:Laplacians and random walks}, which immediately yields the proof of Theorem \ref{thm:comparisontheoremnetworkflow} in Section \ref{sec:Optimal Network Flow}.

\begin{corollary}
\label{cor:lap and rws}
For $u,v\in V$, $f=(f_z)_{z\in V}\in\mathbb{R}^V$ with $\mathbb{1}^Tf=0$,
\begin{align*}
	(e_{u}-e_{v})^T L^{+} f
	&=
	\sum_{z\in V} \sum_{t=0}^\infty (P^t_{uz} - P^t_{vz}) \frac{f_z}{d_z}.
\end{align*}
\end{corollary}

\begin{proof}
From Lemma \ref{lem:Laplacians and random walks}, by summing the quantity $(e_{u}-e_{v})^T L^{+} (e_{w}-e_{z})$ over $z\in V$, recalling that $\sum_{z\in V} e_z = \mathbb{1}$ and $L^+\mathbb{1} = 0$ we have
$
	(e_{u}-e_{v})^T L^{+} e_{w}
	=
	\sum_{t=0}^\infty (P^t_{uw} - P^t_{vw}) \frac{1}{d_w}
	- \frac{1}{|V|} \sum_{z\in V} \sum_{t=0}^\infty (P^t_{uz} - P^t_{vz}) \frac{1}{d_z}.
$
The identity in the statement of the Lemma follows easily as $f=\sum_{w\in V} f_w e_w$ and $\sum_{w\in V} f_w=0$ by assumption. 
\end{proof}


\section{Decay of correlation}
\label{app:Decay of correlation}

This appendix is devoted to the proofs of the decay of correlation properties stated in Section \ref{sec:Optimal Network Flow}, namely, Theorem \ref{thm:Decay of correlation} (set-to-point) and Lemma \ref{lem:Point-to-set decay of correlation} (point-to-set).
These proofs rely on the sensitivity analysis for the network flow problem established in Theorem \ref{thm:comparisontheoremnetworkflow}.

Henceforth, consider the general setting introduced in Appendix \ref{sec:Laplacians and random walks}, and let $d$ denote the graph-theoretical distance on $G$: that is, $d(u,v)$ denotes the length of the shortest path between vertex $u$ and vertex $v$. Note that $d(u,v) = \inf\{t\ge 0 : P^t_{uv}\neq 0\}$, as we assumed that to each edge $\{v,w\}\in E$ is associated a non-negative weight $W_{vw}=W_{wv}>0$. Let $n:=|V|$, and let $-1\le\lambda_n \le \lambda_{n-1} \le \cdots \le \lambda_2 < \lambda_1 =1$ be the eigenvalues of $P$. Define $\lambda:=\max\{|\lambda_2|, |\lambda_n|\}$. 

The backbone behind the proof of Theorem \ref{thm:Decay of correlation} and Lemma \ref{lem:Point-to-set decay of correlation} is given by the following lemma.

\begin{lemma}
\label{lem:Laplacians and random walks bound}
For any $U, Z\subseteq V$ and $(f_z)_{z\in Z}\in\mathbb{R}^Z$ we have
$$
	\sqrt{
	\frac{1}{2} \!\!\sum_{\substack{u,v\in U :\\\{u,v\}\in E}}\!\!\!
	\bigg( \sum_{z\in Z} \sum_{t=0}^\infty (P^t_{uz} \!-\! P^t_{vz}) f_z \bigg)^2
	}
	\!\!\le\! \alpha\,\!
	\frac{\lambda^{d(U,Z)}}{1-\lambda}
	\!\sqrt{\sum_{z\in Z} f_z^2 d_z},
$$
with $\alpha:= \frac{\max_{v\in U} \sqrt{2 |\mathcal{N}(v) \cap U|}}{\min_{v\in U} \sqrt{d_v}}$.
\end{lemma}

\begin{proof}
Let $\Gamma := D^{1/2} P D^{-1/2} = D^{-1/2} W D^{-1/2}$. This matrix is clearly similar to $P$ and symmetric. Let denote by $\psi_n,\ldots,\psi_1$ the orthonormal eigenvectors of $\Gamma$ corresponding to the eigenvalues $\lambda_n \le \lambda_{n-1} \le \cdots \le \lambda_2 \le \lambda_1$, respectively. By substitution, it is easy to check that $\sqrt{\pi}\equiv(\sqrt{\pi_v})_{v\in V}$ is an eigenvector of $\Gamma$ with eigenvalue equal to $1$, where we recall that $\pi_v=d_v/\sum_{v\in V}d_v$. Since this eigenvector has positive entries, it follows by the Perron-Frobenius theory that $-1 \le \lambda_n \le \lambda_{n-1} \le \cdots \le \lambda_2 < \lambda_1 = 1$ and that $\psi_1=\sqrt{\pi}$. As $\Gamma=\sum_{k=1}^n \lambda_k \psi_k\psi_k^T$, by the orthonormality of the eigenvectors we have, for $t\ge 0$, $u,z\in V$,
$$
	P^t_{uz} = (D^{-1/2}\Gamma^t D^{1/2})_{uz}
	= \pi_z + \sum_{k=2}^n \lambda^t_k \psi_{ku} \psi_{kz} \sqrt{\frac{d_z}{d_u}},
$$
where $\psi_{ku}\equiv (\psi_k)_u$ is the $u$-th component of $\psi_k$.
As $P^t_{uz}=0$ if $d(u,z)>t$, we have $P^t_{uz}-P^t_{vz} = \mathbf{1}_{d(U,Z) \le t} (P^t_{uz}-P^t_{vz})$ for any $u,v\in U, z\in Z$. Hence, for any $u,v\in U$, let $g_{uv} := \sum_{z\in Z} \sum_{t=0}^\infty (P^t_{uz} - P^t_{vz}) f_z$, which equals
\begin{align*}
	g_{uv} 
	= 
	\sum_{k=2}^n
	\left(
	\frac{\psi_{ku}}{\sqrt{d_u}}
	-
	\frac{\psi_{kv}}{\sqrt{d_v}}
	\right)
	\sum_{z\in Z} 
	\psi_{kz} \sqrt{d_z} f_z
	\sum_{t=d(U,Z)}^\infty \lambda^t_k.
\end{align*}
As $\lambda <1$ by assumption, the geometric series converges for any $k\neq 1$. If we define 
$h_{u}:=
	\sum_{k=2}^n
	\frac{\lambda_k^{d(U,Z)}}{1-\lambda_k}
	\frac{\psi_{ku}}{\sqrt{d_u}}
	\sum_{z\in Z} \psi_{kz} \sqrt{d_z} f_z
$	
for each $u\in V$, we have $g_{uv} = h_u - h_v$, and 
the triangle inequality for the $\ell_2$-norm yields
$
	(\sum_{u,v\in U :\{u,v\}\in E} g_{uv}^2)^{1/2}
	\le 
	2 (\sum_{u,v\in U :\{u,v\}\in E} 
	h_{u}^2
	)^{1/2}
$
which is upper-bounded by
$
	2 \sqrt{\max_{u\in U} |\mathcal{N}(u) \cap U|} (\sum_{u\in U} h_u^2)^{1/2},
$
where the factor $2$ comes by the symmetry between $u$ and $v$. Expanding the squares and using that $|\lambda_k| \le \lambda$ for each $k\neq 1$, we get
\begin{align*}
	d_u h_{u}^2
	\le&\,
	\frac{\lambda^{2d(U,Z)}}{(1-\lambda)^2}
	\sum_{k=1}^n
	\psi^2_{ku}
	\bigg(\sum_{z\in Z} \psi^2_{kz} d_z f^2_z
	+ \sum_{z\neq z'} \psi_{kz}\psi_{kz'} \sqrt{d_zd_{z'}} f_zf_{z'} \bigg)\\
	&\,+ \sum_{k\neq k'\ge 2}
	\frac{\lambda_k^{d(U,Z)}\lambda_{k'}^{d(U,Z)}}{(1-\lambda_k)(1-\lambda_{k'})}
	\psi_{ku}\psi_{k'u}
	\sum_{z,z'} \psi_{kz} \psi_{k'z'} \sqrt{d_zd_{z'}} f_z f_{z'},
\end{align*}
where we also used that $\sum_{k=2}^n x_k \le \sum_{k=1}^n x_k$ if $x_1,\ldots,x_n$ are non-negative numbers.
Let $\Psi$ denote the matrix having the eigenvectors $\psi_1,\ldots,\psi_n$ in its columns, namely, $\Psi_{uk}=(\psi_{k})_u=\psi_{ku}$. This is an orthonormal matrix, so both its columns and rows are orthonormal, namely, $\sum_{u=1}^n \psi_{ku} \psi_{k'u} = \mathbf{1}_{k=k'}$ and $\sum_{k=1}^n \psi_{ku} \psi_{kv} = \mathbf{1}_{u=v}$. Using this fact,
$$
	\sum_{u\in U} h^2_u \le 
	\frac{\sum_{u\in V} d_u h^2_u}{\min_{u\in U} d_u}
	\le 
	\frac{1}{\min_{u\in U} d_u} \frac{\lambda^{2d(U,Z)}}{(1-\lambda)^2}
	\sum_{z\in Z} d_z f_z^2,
$$
and the proof follows by putting all the pieces together, realizing that the upper bound in the statement of the lemma corresponds to $\frac{1}{\sqrt{2}}(\sum_{u,v\in U :\{u,v\}\in E} g_{uv}^2)^{1/2}$.
\end{proof}

%
%

We are finally ready to present the proof of Theorem \ref{thm:Decay of correlation}.

\begin{proof}[Proof of Theorem \ref{thm:Decay of correlation}]
Consider the setting developed in Section \ref{sec:Optimal Network Flow}. Fix $\varepsilon\in\mathbb{R}$. From Theorem \ref{thm:comparisontheoremnetworkflow} we have
$
	\frac{d x^\star(b(\varepsilon))}{d \varepsilon} 
	= \Sigma(b(\varepsilon))A^T L(b(\varepsilon))^{+} \frac{d b(\varepsilon)}{d \varepsilon}
$
or, entry-wise, for any $(u,v)\in \vec{E}$,
$$
	\frac{d x^\star(b(\varepsilon))_{(u,v)}}{d \varepsilon} 
	= W(b(\varepsilon))_{uv} (e_u-e_v)^T L(b(\varepsilon))^{+} \frac{d b(\varepsilon)}{d \varepsilon}.
$$
Define $U:=V_{\vec F}$, and let $(U,F)$ be the undirected graph associated to $(U,\vec{F})$ (see Remark \ref{rem:notation}). Clearly, $(\sum_{e\in \vec{F}} (\frac{d x^\star(b(\varepsilon))_e}{d \varepsilon})^2)^{1/2}$ is upper-bounded by
\begin{align}
	\gamma(b(\varepsilon))
	\sqrt{
	\frac{1}{2} \!\sum_{u,v\in U :\{u,v\}\in E}	
	\!\!\bigg(\!
	(e_{u}\!-\!e_{v})^T L(b(\varepsilon))^{+} \frac{d b(\varepsilon)}{d \varepsilon}
	\!\bigg)^2
	},\!\!
	\label{turning point}
\end{align}
where $\gamma(b):=\max_{u,v\in U} W(b)_{uv}$ for any $b\in\operatorname{Im}(A)$.
Corollary \ref{cor:lap and rws} yields (choosing $f=\frac{d b(\varepsilon)}{d \varepsilon}$, using that $\mathbb{1}^Tf=0$ and $f_v\neq 0$ if and only if $v\in Z$) that $(e_{u}-e_{v})^T L(b(\varepsilon))^{+} \frac{d b(\varepsilon)}{d \varepsilon}$ equals
$
	\sum_{z\in Z} \sum_{t=0}^\infty (P(b(\varepsilon))^t_{uz} \!-\! P(b(\varepsilon))^t_{vz}) 	
	\frac{1}{d(b(\varepsilon))_z} \frac{d b(\varepsilon)_z}{d \varepsilon},
$
so that \eqref{turning point} reads
$$
	\gamma(b(\varepsilon))
	\sqrt{
	\frac{1}{2} \sum_{u,v\in U :\{u,v\}\in E}
	\bigg(
	\sum_{z\in Z} \sum_{t=0}^\infty (P(b(\varepsilon))^t_{uz} - P(b(\varepsilon))^t_{vz}) 	
	\frac{1}{d(b(\varepsilon))_z} \frac{d b(\varepsilon)_z}{d \varepsilon}
	\bigg)^2
	}.
$$
Lemma \ref{lem:Laplacians and random walks bound} yields (choosing $f_z=\frac{1}{d_z} \frac{d b(\varepsilon)_z}{d \varepsilon}$) that the previous quantity is upper-bounded by
$$
	\gamma(b(\varepsilon))
	\,\alpha(b(\varepsilon))\,
	\frac{\lambda(b(\varepsilon))^{d(U,Z)}}{1-\lambda(b(\varepsilon))}
	\sqrt{\sum_{z\in Z} \bigg(\frac{d b(\varepsilon)_z}{d \varepsilon}\bigg)^2 \frac{1}{d(b(\varepsilon))_z}},
$$
with $\alpha(b):= \frac{\max_{v\in U} \sqrt{2 |\mathcal{N}(v) \cap U|}}{\min_{v\in U} \sqrt{d(b)_v}}$ for any $b\in\operatorname{Im}(A)$. Combining everything together, we obtain
$$
	\left\| \frac{d x^\star(b(\varepsilon))}{d \varepsilon} \right\|_{\vec{F}}
	\le c(b(\varepsilon))\,
	\frac{\lambda(b(\varepsilon))^{d(U,Z)}}{1-\lambda(b(\varepsilon))}
	\left\|\frac{d b(\varepsilon)}{d \varepsilon}\right\|_Z,
$$
where $c(b):= \frac{\max_{v\in U} \sqrt{2 |\mathcal{N}(v) \cap U|}}{\min_{v\in U} d(b)_v} \gamma(b)$ for $b\in\operatorname{Im}(A)$. The proof follows by taking the supremum over $b\in\operatorname{Im}(A)$.
\end{proof}

%

The proof of Lemma \ref{lem:Point-to-set decay of correlation} follows analogously from the proof of Theorem \ref{thm:Decay of correlation}, upon exploiting the symmetry of the pseudoinverse of the graph Laplacian.

\section{Localized algorithm}
\label{app:error localized}

This section is devoted to the proof of Theorem \ref{thm:error localized}, which states error bounds for the localized projected gradient descent algorithm. The proof relies on the decay of correlation property established in Theorem \ref{thm:Decay of correlation} for the network flow problem. Recall that the constants appearing in the bounds in Theorem \ref{thm:error localized} do not depend on the choice of the subgraph $\vec G'$ of $\vec G$, but depend only on $\mu$, $Q$, $k_+$, and $k_-$. To prove this type of bounds, we first need to develop estimates to relate the eigenvalues of weighted subgraphs to the eigenvalues of the corresponding unweighted graph.

\subsection{Eigenvalues interlacing}
\label{sec:Eigenvalues Interlacing}
Let $G=(V,E)$ be a simple (i.e., no self-loops, and no multiple edges), connected, undirected graph, with vertex set $V$ and edge set $E$. Let $B\in\mathbb{R}^{V\times V}$ be the vertex-vertex adjacency matrix of the graph, which is the symmetric matrix defined as
 $B_{uv}:=1$ if $\{u,v\}\in E$, $B_{uv}:=0$ otherwise.
If $n:=|V|$, denote by $\mu_{n} \le \mu_{n-1} \le \cdots \le \mu_2 \le \mu_1$ the eigenvalues of $B$.
Let $G'=(V', E')$ be a connected subgraph of $G$. Assume that to each edge $\{u,v\}\in E'$ is associated a non-negative weight $W_{uv}=W_{vu}>0$, and let $W_{uv}=0$ if $\{u,v\}\not\in E$. Let $D'$ be a diagonal matrix with entries $D'_{vv}=\sum_{w\in V'} W'_{vw}$ for each $v\in V'$. Let $P':=D'^{-1}W'$. If $m:=|V'|$, denote by $\lambda'_{m} \le \lambda'_{m-1} \le \cdots \le \lambda'_2 \le \lambda'_1$ the eigenvalues of $P'$. The following proposition relates the eigenvalues of $P'$ to the eigenvalues of $B$. In particular, we provide a bound for the second largest eigenvalue in magnitude of $P'$ with respect to the second largest eigenvalue in magnitude of $B$, uniformly over the choice of $G'$.

\begin{proposition}[Eigenvalues interlacing]
\label{prop:interlacing}
Let $w_- \le W_{vw}\le w_+$ for any $\{v,w\}\in E$, for some constants $w_-,w_+>0$. Let $k_-$ and $k_+$ be, resp., the min and max degree of $G$. Then,
$$
	1 - \frac{w_+k_+}{w_-k_-} + \frac{w_+}{w_-k_-} \mu_{i+n-m}
	\le
	\lambda'_i
	\le 1 - \frac{w_-k_-}{w_+k_+} + \frac{w_-}{w_+k_+} \mu_i.
$$
Therefore, if $\lambda':=\max\{|\lambda'_2|,|\lambda'_{m}|\}$ and $\mu:=\max\{|\mu_2|,|\mu_{n}|\}$, we have
$
	\lambda' \le \frac{w_+k_+}{w_-k_-} -1 + \frac{w_+}{w_- k_-}\mu.
$
\end{proposition}

\begin{proof}
Consider $\Gamma' := D'^{1/2} P' D'^{-1/2} = D'^{-1/2} W' D'^{-1/2}$. As this matrix is similar to $P'$, it shares the same eigenvalues with $P'$. Let $L':=D'-W'$ be the Laplacian associated to $G'$. The Courant-Fischer Theorem yields
$$
	\lambda'_i = \max_{\substack{S\subseteq\mathbb{R}^m\\\operatorname{dim}(S)=i}} \min_{x\in S} \frac{x^T \Gamma' x}{x^Tx}
	= 1 + \max_{\substack{S\subseteq\mathbb{R}^m\\\operatorname{dim}(S)=i}} \min_{y\in S} \frac{-y^T L' y}{y^TD'y},
$$
where we used that $x^T \Gamma' x= x^Tx - y^TL'y$ with $y:=D'^{-1/2}x$, and that the change of variables $y=D'^{-1/2}x$ is non-singular (note that as $G'$ is connected, then $D'$ has non-zero entries on the diagonal). The Laplacian quadratic form yields
$
	y^T L' y = \frac{1}{2} \sum_{u,v\in V'} W_{uv} (y_u-y_v)^2,
$
which is upper-bounded by
$
	w_+ \frac{1}{2} \sum_{u,v\in V'} B_{uv} (y_u-y_v)^2
	= w_+ y^T \mathcal{L}' y,
$
where $\mathcal{L}'$ is the Laplacian of the \emph{unweighted} graph $G'=(V',E')\equiv(V',E',B')$ with $B':=B_{V',V'}$. Note that we have $\mathcal{L}'=K'-B'$, where $K'$ is diagonal and $K'_{vv}=\sum_{w\in V'}B'_{vw}$ is the degree of vertex $v\in V'$ in $G'$. As $y^TK'y = \sum_{v\in V'} K_{vv} y^2_v \le k_+ y^Ty$, we have
$
	y^T L' y
	\le w_+ k_+ y^T y - w_+ y^T B' y.
$
At the same time, $y^TD'y \ge w_-k_- y^Ty$. Therefore,
$
	\lambda'_i
	\ge 1 - \frac{w_+k_+}{w_-k_-} + \frac{w_+}{w_-k_-} \max_{\substack{S\subseteq\mathbb{R}^m\\\operatorname{dim}(S)=i}} \min_{y\in S} \frac{y^T B' y}{y^Ty},
$
and by the Courant-Fischer Theorem the right-hand side equals $1 - \frac{w_+k_+}{w_-k_-} + \frac{w_+}{w_-k_-} \mu'_i$, where $\mu'_{n} \le \mu'_{n-1} \le \cdots \le \mu'_2 \le \mu'_1$ are the eigenvalues of $B'$. Analogously, it is easy to prove that
$
	\lambda'_i
	\le 1 - \frac{w_-k_-}{w_+k_+} + \frac{w_-}{w_+k_+} \mu'_i.
$
As $B'$ is a principal submatrix of $B$, the eigenvalue interlacing theorem for symmetric matrices yields $\mu_{i+n-m}\le\mu'_i\le \mu_i$, and we have
$
	\alpha + \beta \mu_{i+n-m}
	\le
	\lambda'_i
	\le \gamma + \delta \mu_i,
$
with $\alpha:=1 - \frac{w_+k_+}{w_-k_-}, \beta:=\frac{w_+}{w_-k_-}, \gamma := 1 - \frac{w_-k_-}{w_+k_+},$ and $\delta:=\frac{w_-}{w_+k_+}$. Clearly,
$
	|\lambda'_i|
	\le -\alpha + \beta \max\{ |\mu_{i+n-m}|, |\mu_i| \},
$
so
$
	\max\{|\lambda'_2|,|\lambda'_{m}|\}
	\le -\alpha \!+\! \beta \max\{
	|\mu_{2+n-m}|, |\mu_2|,|\mu_{n}|, |\mu_m|
	\}
	= -\alpha \!+\! \beta \max\{|\mu_2|,|\mu_{n}|\}.
$
\end{proof}

\subsection{Proof of Theorem \ref{thm:error localized}}

We now present the proof of Theorem \ref{thm:error localized}. The proof relies on repeatedly applying Theorem \ref{thm:Decay of correlation} in Section \ref{sec:Optimal Network Flow} (which captures the decay of correlation for the network flow problem) and the fundamental theorem of calculus.

\begin{proof}[Proof of Theorem \ref{thm:error localized}]
Consider the setting of Section \ref{sec:Scale-free algorithm}.\\
\textbf{Bias term.} Let us first bound the bias outside $\vec E'$. Let $n:=|V|$, and for each $b\in\operatorname{Im}(A)$ let $-1\le\lambda_n(b) \le \lambda_{n-1}(b) \le \cdots \le \lambda_2(b) < \lambda_1(b) =1$ be the eigenvalues of $P(b)$. Let $\lambda(b):=\max\{|\lambda_2(b)|, |\lambda_n(b)|\}$ and $\lambda:=\sup_{b\in\operatorname{Im}(A)} \lambda(b)$. Define $b(\varepsilon):=b+\varepsilon p$, for any non-negative real number $\varepsilon \ge 0$. If $e\in \vec E'^C$, then $T'_{b(\varepsilon)}(x^\star(b))_e = x^\star(b)_e$ and
\begin{align*}
	\operatorname{Bias}(\vec G')_e
	&= x^\star(b(1))_e - x^\star(b(0))_e
	= \int_0^1 d\varepsilon \, 
	\frac{d x^\star(b(\varepsilon))_e}{d \varepsilon}.
\end{align*}
By the triangle inequality for the $\ell_2$-norm and Theorem \ref{thm:Decay of correlation},
$$
	\| \operatorname{Bias}(\vec G') \|_{\vec E'^C}
	\le \sup_{\varepsilon\in\mathbb{R}} \left\| \frac{d x^\star(b(\varepsilon))}{d \varepsilon} \right\|_{\vec E'^C}
	\!\le c
	\| p \|
	\frac{\lambda^{d(\Delta(\vec G'),Z)}}{1-\lambda},
$$
where we used that $\sup_{\varepsilon\in\mathbb{R}} \|\frac{d b(\varepsilon)}{d \varepsilon}\|_Z= \|p\|$, as $\frac{d b(\varepsilon)_v}{d\varepsilon}=p_v$ for $v\in Z$ and $\frac{d b(\varepsilon)_v}{d\varepsilon}=0$ for $v\not\in Z$, and $c:=\sqrt{2k_+} Q/k_-$.

Let us now consider the bias inside $\vec E'$. Consider problem \eqref{def:problem_localized} in Section \ref{sec:Scale-free algorithm}.
For any $\varepsilon>0,\theta>0$, define
\begin{align}
	b'(\varepsilon,\theta)
	:= b(\varepsilon)_{V'} - A_{V',\vec E'^C} x^\star(b(\theta))_{\vec E'^C}.
	\label{def:b large}
\end{align}
Without loss of generality, we can index the elements of $V'$ and $\vec E'$ so that the matrix $A$ has the following structure:
$$
	A
	=
	\left( \begin{array}{cc}
	A_{V',\vec E'} & A_{V',\vec E'^C} \\
	A_{V'^C,\vec E'} & A_{V'^C,\vec E'^C}
	\end{array} \right)
	=
	\left( \begin{array}{cc}
	A' & A_{V',\vec E'^C} \\
	\mathbb{0} & A_{V'^C,\vec E'^C}
	\end{array} \right).
$$
For any $x$ that satisfies the flow constraints on $\vec E'^C$ with respect to $b(\varepsilon)$, $A_{V\setminus V', \vec E'^C}x_{\vec E'^C}=b(\varepsilon)_{V\setminus V'}$, we have
\begin{align}
	\left\{ \lim_{t\rightarrow\infty} T'^t_{b+p}(x) \right\}_{\vec E'}
	&= x'^\star( b(1)_{V'} - A_{V',\vec E'^C} x_{\vec E'^C} ).\nonumber
\end{align}
Clearly $x^\star(b)$ satisfies the flow constraints on $\vec E'^C$ with respect to $b(1)$, as $p$ is supported on $V'$ so that $b(\varepsilon)_{V'^C} = b_{V'^C}$. Recalling the definition of $b'(\varepsilon,\theta)$ in \eqref{def:b large}, we then have
$
	(\lim_{t\rightarrow\infty} T'^t_{b+p}(x^\star(b)))_{\vec E'}
	= x'^\star(b'(1,0)).
$
On the other hand, as $x^\star(b(1))$ is a fixed point of the map $T'_{b(1)}$, we can characterize the components of $x^\star(b(1))$ supported on $\vec E'$ as
\begin{align*}
	x^\star(b(1))_{\vec E'}
	&=
	\left\{\lim_{t\rightarrow\infty} T'^t_{b(1)}(x^\star(b(1))\right\}_{\vec E'}
	= x'^\star(b'(1,1)).
\end{align*}
It is easy to check that $b'(\varepsilon,\theta)\in\operatorname{Im}(A')$ for each value of $\varepsilon$ and $\theta$. In fact, as $\vec G'$ is connected by assumption, then $\operatorname{Im}(A')$ corresponds to the subspace of $\mathbb{R}^{V'}$ orthogonal to the all-ones vector $\mathbb{1}$. We have
$
	\mathbb{1}^T b'(\varepsilon,\theta)
	= \mathbb{1}^Tb_{V'} + \varepsilon \mathbb{1}^Tp_{V'} - \mathbb{1}^TA_{V',\vec E'^C} x^\star(b(\theta))_{\vec E'^C}.
$
Note that $\mathbb{1}^Tp_{V'} = 0$ by assumption. Also, $0=\mathbb{1}^Tb = \mathbb{1}^Tb_{V'} + \mathbb{1}^Tb_{V'^C}$ so that $\mathbb{1}^Tb_{V'} = -\mathbb{1}^Tb_{V'^C}$. Analogously, as $\mathbb{1}^TA=\mathbb{0}^T$, we have $\mathbb{1}^TA_{V',\vec E'^C} = - \mathbb{1}^TA_{V'^C,\vec E'^C}$.
Hence, 
$$
	\mathbb{1}^T b'(\varepsilon,\theta)
	= -\mathbb{1}^Tb_{V'^C} + \mathbb{1}^TA_{V'^C,\vec E'^C} x^\star(b(\theta))_{\vec E'^C}
	= \mathbb{0}^T,
$$
where the last equality follows as clearly $A_{V'^C,\vec E'^C} x^\star(b(\theta))_{\vec E'^C} = b_{V'^C}$. Therefore, for $e\in \vec E'$,
\begin{align*}
	\operatorname{Bias}(\vec G')_e
	= \int_0^1 d\theta  
	\, 
	\frac{d x'^\star(b'(1,\theta))_e}{d \theta}.
\end{align*}
For each $b'\in\operatorname{Im}(A')$, let $W'(b')\in\mathbb{R}^{V'\times V'}$ be a symmetric matrix defined as $W'(b')_{uv}=(\frac{\partial^2 f_e(x'^\star(b')_e)}{\partial x_e^2})^{-1}$ if either $e=(u,w)\in \vec E$ or $e=(w,u) \in \vec E$, and $W'(b')_{uv}:=0$ otherwise.
Let $D'(b')\in\mathbb{R}^{V'\times V'}$ be a diagonal matrix with entries $D'(b')_{vv}=\sum_{w\in V'} W'(b')_{vw}$. Let $P'(b'):=D'(b')^{-1}W'(b')$. If $m:=|V|$, let $-1\le\lambda'_m(b') \le \lambda'_{m-1}(b') \le \cdots \le \lambda'_2(b') < \lambda'_1(b') =1$ be the eigenvalues of $P'(b')$ (where this characterization holds as $G'$ is connected by assumption). Define $\lambda'(b'):=\max\{|\lambda'_2(b')|, |\lambda'_m(b')|\}$ and $\lambda':=\sup_{b'\in\operatorname{Im}(A')} \lambda'(b')$. 
Proceeding as above, applying Theorem \ref{thm:Decay of correlation} to the optimization problem defined on $\vec G'$ we get
$$
	\| \operatorname{Bias}(\vec G') \|_{\vec E'}
	\le \sup_{\theta\in\mathbb{R}} \left\| \frac{d x'^\star(b'(1,\theta))}{d \theta} \right\|_{\vec E'},
$$
which is upper-bounded by $c \, \frac{1}{1-\lambda'}
	\sup_{\theta\in\mathbb{R}} 
	\| \frac{\partial b'(1,\theta)}{\partial\theta} \|_{\Delta(\vec G')}$,
where we used that $\frac{\partial b'(\varepsilon,\theta)_v}{\partial\theta}=0$ if $v\in V'\setminus \Delta(\vec G')$, and clearly $d(V',\Delta(\vec G'))=0$ as $\Delta(\vec G')\subseteq V'$.
For $v\in \Delta(\vec G')$ we have
$
	\frac{\partial b'(\varepsilon,\theta)_v}{\partial\theta}
	=
	- \sum_{e\in \vec E'^C} A_{ve} 
	\frac{d x^\star(b(\theta))_{e}}{d \theta}.
$
If $\vec F(v):=\{e\in \vec E'^C: e=(u,v) \text{ or } e=(v,u) \text{ for some } u\in V'^C\}$ denotes the set of edges that are connected to $v$ but do not belong to $\vec E'$, we have
$
	(\frac{\partial b'(1,\theta)_v}{\partial\theta})^2
	\le
	(\sum_{e\in \vec F(v)}  
	|\frac{d x^\star(b(\theta))_{e}}{d \theta}|
	)^2,
$
which by Jensen's inequality admits
$
	|\vec F(v)|
	\sum_{e\in \vec F(v)}
	(\frac{d x^\star(b(\theta))_{e}}{d \theta})^2
$
as a upper bound.
As $\max_{v\in \Delta(\vec G')}|\vec F(v)|\le k_+-1$, applying Theorem \ref{thm:Decay of correlation} as above we get
$$
	\| \frac{\partial b'(1,\theta)}{\partial\theta} \|_{\Delta(\vec G')}
	\le \sqrt{k_+-1}
	\| \frac{d x^\star(b(\theta))}{d\theta} \|_{\vec E'^C}
	\le
	c\sqrt{k_+-1}\,\| p \|\,
	\frac{\lambda^{d(\Delta(\vec G'),Z)}}{1-\lambda}.
$$
Therefore,
$
	\| \operatorname{Bias}(\vec G') \|_{\vec E'}
	\le
	c^2\sqrt{k_+-1}\,\| p \|\,
	\frac{\lambda^{d(\Delta(\vec G'),Z)}}{(1-\lambda')(1-\lambda)}.
$
By the triangle inequality for the $\ell_2$-norm, $\|\operatorname{Bias}(\vec G')\|\le \|\operatorname{Bias}(\vec G')\|_{\vec E'} + \|\operatorname{Bias}(\vec G')\|_{\vec E'^C}$, so we obtain
$$
	\| \operatorname{Bias}(\vec G') \|
	\le 
	c
	( 1
	+
	c\sqrt{k_+-1}
	)
	\left\| p \right\|
	\frac{\lambda^{d(\Delta(\vec G'),Z)}}{(1-\lambda')(1-\lambda)}.
$$
By Proposition \ref{prop:interlacing} we have
$
	\max\{\lambda,\lambda'\} \le \frac{Qk_+}{k_-} -1 + \frac{Q}{k_-} \mu,
$
and the bound for the bias term follows.

\textbf{Variance term.} As
$
	(\lim_{t\rightarrow\infty} T'^t_{b+p}(x^\star(b)))_{\vec E'}
	= x'^\star(b'(1,0)),
$
we get
\begin{align*}
	\| \operatorname{Var}(\vec G',t) \|_{\vec E'}
	&= \
	\| x'^\star(b'(1,0)) - T'^t_{b+p}(x^\star(b))_{\vec E'} \|\\
	&\le e^{- \frac{t}{2Q}} \| x'^\star(b'(1,0)) - x'^\star(b'(0,0)) \|,
\end{align*}
where in the last inequality we used that $x^\star(b)_{\vec E'}=x'^\star(b'(0,0))$. For each $e\in \vec E'$ we have
\begin{align*}
	x'^\star(b'(1,0))_e - 
	x'^\star(b'(0,0))_e
	= \int_0^1 d\varepsilon
	\, 
	\frac{d x'^\star(b'(\varepsilon,0))_e}{d \varepsilon},
\end{align*}
and using the triangle inequality for the $\ell_2$-norm, applying Theorem \ref{thm:Decay of correlation} to the optimization problem defined on $\vec G'$,
$$
	e^{\frac{t}{2Q}} \| \operatorname{Var}(\vec G',t) \|_{\vec E'}
	\!\le \sup_{\varepsilon\in\mathbb{R}} \left\| \frac{d x'^\star(b'(\varepsilon,0))}{d \varepsilon} \right\|_{\vec E'}
	\!\!\!\le \!c
	\| p \|
	\frac{1}{1-\lambda'},
$$
where we used that $\frac{\partial b'(\varepsilon,0)_v}{\partial \varepsilon}=\frac{d b(\varepsilon)_v}{d \varepsilon}=p_v$ for $v\in Z$ and $\frac{d b(\varepsilon)_v}{d\varepsilon}=0$ for $v\not\in Z$, and that $d(V',Z)=0$ as $Z\subseteq V'$. Clearly, $\operatorname{Var}(\vec G',t)_e=0$ for $e\in \vec E'^C$, as $T'_b(x^\star(b))_e = x^\star(b)_e$. Hence, $\| \operatorname{Var}(\vec G',t) \|=\| \operatorname{Var}(\vec G',t) \|_{\vec E'}$ and the proof is concluded as $\lambda' \le \frac{Qk_+}{k_-} -1 + \frac{Q}{k_-} \mu$ by Proposition \ref{prop:interlacing}.
\end{proof}

\end{document}